\documentclass[9pt]{article}
\usepackage{latexsym,amsfonts,euscript,amsthm}

\usepackage{amssymb}
\usepackage{stmaryrd}
\usepackage{mathrsfs,amsmath}
\usepackage{leftidx}
\usepackage{mathabx}

\usepackage[figuresright]{rotating}
\usepackage[table]{xcolor}
\usepackage{multirow}
\usepackage{booktabs}
\usepackage{tabularx}

\usepackage{tocbibind} 

\usepackage[pagebackref]{hyperref}

\hypersetup{%
  colorlinks=true,
linktoc=all,
  citebordercolor=green,
  linkbordercolor=red,
 citecolor=green,
    linkcolor=red,
  urlcolor=blue
}

\newtheorem{lem}{\bf Lemma}[section]

\newtheorem{prop}[lem]{\bf Proposition}

\newtheorem{thm}[lem]{\bf Theorem}
\newtheorem{rmk}[lem]{\bf Remark}
\newtheorem*{rmkA}{\bf Remark}

\newtheorem{cor}[lem]{\bf Corollary}
\newtheorem{hy}[lem]{\bf Hypothesis}

\newtheorem{mainthm}{Theorem}

\newcommand{\ackname}{Acknowledgements}
\makeatletter
\if@titlepage
  \newenvironment{acknowledgement}{%
    \titlepage
    \null\vfil
    \@beginparpenalty\@lowpenalty
    \begin{center}%
      \bfseries \ackname
      \@endparpenalty\@M
    \end{center}}%
  {\par\vfil\null\endtitlepage}
\else
  \newenvironment{acknowledgement}{%
    \if@twocolumn
      \section*{\ackname}%
    \else
      \small
      \begin{center}%
        {\bfseries \ackname\vspace{-0.5em}\vspace{\z@}}%
      \end{center}%
      \quotation
    \fi}
    {\if@twocolumn\else\endquotation\fi}
\fi
\makeatother

\makeatletter 
\@addtoreset{equation}{section}
\makeatother  

\makeatletter
\def\thanks#1{\protected@xdef\@thanks{\@thanks\protect\footnotetext{#1}}}

				   \AtEndDocument{\bigskip{\footnotesize
				   \textsc{Department of Mathematics, Suzhou University of Technology, No.99 South Third Ring
				   Road,
				   Changshu, Jiangsu, 215500, China.} \par  
				   \textit{Email address}: \href{mailto:yuzeng2004@163.com}{yuzeng2004@163.com} \par
				   \addvspace{\medskipamount}
				   \textsc{Department of Mathematics, Khoy.C., Islamic Azad University, Khoy, Iran.} \par  
				   \textit{Email address}: \href{mailto:me.ghaffarzadeh@iau.ac.ir}{me.ghaffarzadeh@iau.ac.ir}   \par
				   \addvspace{\medskipamount}
				    \textsc{Department of Mathematics, Urmia University, Urmia, Iran.} \par  
				   \textit{Email address}: \href{mailto:m.ghasemi@urmia.ac.ir}{m.ghasemi@urmia.ac.ir} \par
				      \addvspace{\medskipamount}
				    \textsc{Department of Mathematics, Suzhou University of Technology, No.99 South Third Ring
				   Road,
				   Changshu, Jiangsu, 215500, China.} \par  
				   \textit{Email address}: \href{mailto:dfyang1228@163.com}{dfyang1228@163.com} \par
				 }}

\title{Finite groups of non-prime-power order with exactly four character codegrees
\thanks{
\textbf{2020 MR Subject Classification}\,\, Primary 20C15\\
The first author and the four author are supported by the NSF of China (No. 12171058, 12301018), the Natural Science Foundation of Jiangsu Province (No. BK20231356) and the Natural Science Foundation of the Jiangsu Province Higher Education Institutions of China (No. 23KJB110002).
}}

\author{Yu Zeng, Mehdi Ghaffarzadeh, Mohsen Ghasemi*\thanks{*Corresponding author.}, Dongfang Yang\\
\small{\emph{In memory of Mehdi Ghaffarzadeh who passed away while writing this paper}}
}

\date{}

\begin{document}

\maketitle

\begin{abstract}
For an irreducible complex character $\chi$ of a finite group $G$, the \emph{codegree} of $\chi$ is defined by $|G:\ker(\chi)|/\chi(1)$, where $\ker(\chi)$ is the kernel of $\chi$.
In this paper, we 
give a detailed characterization of finite groups of non-prime-power order 
with exactly four (irreducible) character codegrees.
\end{abstract}

\section{Introduction}

For an irreducible complex character $\chi$ of a finite group $G$,
the \emph{codegree} of $\chi$ is defined as
$$\mathrm{cod}(\chi) =\frac{|G: \ker(\chi)|}{\chi(1)}.$$
This definition was introduced by Qian in \cite{qian2002}
and first systematically studied by Qian, Wang and Wei \cite{qian2007}.

Since the papers by Isaacs and Passman in the 1960s, the influence of the set of (irreducible) character degrees
on the structure of finite groups has been extensively studied.
As a ``dual" concept of the set of character degrees,
the \emph{set of (irreducible) character codegrees} $\mathrm{cod}(G)$ of a finite group $G$ also plays a significant role in
determining the structure of $G$.
In recent years, there has been a growing interest in exploring the structure of finite groups 
with a small number of (irreducible) character codegrees.
Du and Lewis \cite{du2016} demonstrated that a group of prime power order
 with at most three character codegrees has nilpotency class at most $2$.
Alizadeh et al. \cite{alizadeh2019} characterized finite nonnilpotent groups with at most three character codegrees.
Qian and Zeng \cite{qian2023} completed the classification of finite nonnilpotent groups with exactly three character codegrees.
Croome and Lewis \cite{croome2020}, and Moret\'o \cite{moreto2022} characterized groups of prime power order with exactly four character codegrees.
Liu and Yang \cite{liu2021} classified finite nonsolvable groups with exactly four character codegrees.
Recently, Liu and Song \cite{liu2025} characterized certain finite solvable groups with exactly four character codegrees,
and provided a list of possible sets of character codegrees for finite solvable groups 
with exactly four character codegrees.

Building on the classification of finite nonsolvable groups with exactly four character codegrees by Liu and Yang \cite{liu2021}, this paper aims to provide a detailed characterization of finite groups of non-prime-power order that have exactly four character codegrees.

Before stating the main result of this paper, we introduce some notation.
We use the symbol $\mathsf{SmallGroup}(m,i)$ for the $i$-th group of 
the groups of order $m$ in the Small Groups library of $\mathsf{GAP}$ (\cite{gap});
$(\mathsf{C}_{m})^{n}$ for the direct product of $n$ copies of the cyclic group  $\mathsf{C}_{m}$ of order $m$;
$\mathsf{D}_{2^{n}}$, where $n\geq 3$, for the dihedral group of order $2^{n}$;
$\mathsf{Q}_{2^{n}}$, where $n\geq 3$, for the generalized quaternion group of order $2^{n}$;
$\mathsf{ES}(2^{5}_-)$ for the extraspecial $2$-group 
which is a central product of $\mathsf{D}_8$ and 
$\mathsf{Q}_8$;
$\mathbb{F}_{p^n}$, where $p$ is a prime, for the finite field with $p^{n}$ elements.

Throughout the paper, all groups considered are finite and $p,q,r$ always denote primes.

\begin{mainthm}\label{thmA}
   Let $G$ be a finite group of non-prime-power order.
   Then $G$ has exactly four irreducible character codegrees if and only if one of the following holds.
   \begin{description}
	 \item[(1)] $G=P\times Q$ where $P$ is an elementary abelian $p$-group
       and $Q$ is an elementary abelian $q$-group for distinct primes $p$ and $q$.
      \item[(2)] $G$ is a Frobenius group with complement $P \in \mathrm{Syl}_{p}(G)$ and kernel $N\in \mathrm{Syl}_{q}(G)$,
	   and one of the following holds.
	  \begin{description}
		\item[(2a)] $P\cong \mathsf{Q}_8$ and $N\cong (\mathsf{C}_{q})^{2}$.
		\item[(2b)] $P\cong \mathsf{C}_{p^{2}}$, $N\cong (\mathsf{C}_{q})^{td}$ is a homogeneous $P$-module over $\mathbb{F}_q$ where
		 $t$ is a positive integer and $d$ is the multiplicative order of $q$ modulo $|P|$.
		 \item[(2c)] $P\cong \mathsf{C}_{p}$, $N$ is an abelian group of exponent $q^{2}$, and all $G$-chief factors in $N$ are isomorphic as $P$-modules over $\mathbb{F}_q$.
		 \item[(2d)] $P\cong \mathsf{C}_{p}$, $N$ is an elementary abelian $q$-group, and there are exactly two non-isomorphic $P$-modules over $\mathbb{F}_q$ among all $G$-chief factors in $N$. 
		 \item[(2e)] $P\cong \mathsf{C}_{p}$, $N/\ker(\theta^{G})$ is an ultraspecial $q$-group of order $q^{3d}$ for each nonlinear $\theta\in \mathrm{Irr}(N)$ where $d$ is the multiplicative order of $q$ modulo $|P|$.
		 Also, either $N/N'$
			is an abelian group of exponent $q^{2}$ and all $G$-chief factors in $N/N'$ are isomorphic as a $P$-module,
            or $N/N'$ is an elementary abelian $q$-group and there are exactly two non-isomorphic $P$-modules among all $G$-chief factors in $N/N'$.
		 \item[(2f)] $P\cong \mathsf{C}_{p}$, $N$ has nilpotency class at least $2$, and $N/N'$ is a homogeneous $P$-module over $\mathbb{F}_q$. Also, for each nonlinear $\theta\in \mathrm{Irr}(N)$, 
		 there exists a positive integer $k$ such that
		 $|N:\ker(\theta^G)|/\theta(1)=q^{k}>q^{d}$ where $d$ is the multiplicative order of $q$ modulo $|P|$.
	  \end{description}
	  \item[(3)] $G=N \rtimes P$ where $P\cong \mathsf{C}_{p}$ and $N$ is a semi-extraspecial $q$-group such that $N'=\mathbf{Z}(G)$. 
	  Further, $G/N'$ is a Frobenius group with kernel $N/N'$, 
	  and
	  $N/N'$ is a homogeneous $P$-module over $\mathbb{F}_q$. 
	  \item[(4)] $G=N \rtimes P$ such that $\mathbf{C}_{P}(x)$ is a non-normal subgroup of order $2$ of $P$ for each nontrivial $x \in N$, and one of the following holds.
	  \begin{description}
		 \item[(4a)] $P\cong \mathsf{D}_{8}$ and $N\cong (\mathsf{C}_{3})^{2}$.
		 \item[(4b)] $P\cong \mathsf{SmallGroup}(16,13)$ and $N\cong (\mathsf{C}_{5})^{2}$. 
		 \item[(4c)] $P\cong \mathsf{ES}(2^{5}_{-})$ and $N\cong (\mathsf{C}_{3})^{4}$.
	  \end{description}
	  \item[(5)] $G=H\times C$ 
	  where $H$ is a Frobenius group with complement $P_0\cong \mathsf{C}_{p}$ and kernel $N$
	  such that $N$ is a homogeneous $P_0$-module over $\mathbb{F}_q$,
	  and $C$ is an elementary abelian $p$-group.
	  \item[(6)] $G$ has a normal series $1\lhd V\lhd K\lhd G$ such that $G/V$ is a Frobenius group with complement of order $p$ and cyclic kernel $K/V$
	of prime order $q=\frac{r^{pm}-1}{r^m-1}$,
	and $K$ is a Frobenius group with elementary abelian kernel $V$ of order $r^{pm}$ such that $V$ is minimal normal in $K$.
	  \item[(7)] $G$ is isomorphic to $\mathrm{SL}_2(2^{f})$ for $f\geq 2$.
   \end{description}
\end{mainthm}

\begin{rmkA}
	{\rm We list below the $\mathrm{cod}(G)$ for the groups $G$ appearing in (1)-(7) of Theorem \ref{thmA}, 
	together with concrete examples for the two sub-cases (2e) and (2f). Denote by $d$ the multiplicative order of $q$ modulo $|P|$. 
\begin{description}
	\item[$\bullet$] (1) $\mathrm{cod}(G)=\{ 1,p,q,pq \}$;
	(2a) $\mathrm{cod}(G)=\{ 1,2,4,q^{2} \}$;
	(2b) $\mathrm{cod}(G)=\{ 1,p,p^{2},q^{d} \}$;
	(2c)-(2e) $\mathrm{cod}(G)=\{ 1,p,q^{d},q^{2d} \}$;
	(2f) $\mathrm{cod}(G)=\{ 1,p,q^{d}, q^{k}\}$ for some $k>d$;
	(3)  $\mathrm{cod}(G)=\{ 1,p,q^{d}, pq \sqrt{|N:N'|} \}$; 
	(4a) $\mathrm{cod}(G)=\{ 1,2,4, 18 \}$;
	(4b) $\mathrm{cod}(G)=\{ 1,2,8, 50 \}$;
	(4c) $\mathrm{cod}(G)=\{ 1,2,8, 162 \}$;
	(5) $\mathrm{cod}(G)=\{ 1,p,q^{d},pq^{d} \}$;
	(6) $\mathrm{cod}(G)=\{ 1,p,q, pr^{pm} \}$;
	(7) $\mathrm{cod}(G)=\{ 1, 2^{2f}-2^{f},2^{2f}+2^{f}, 2^{2f}-1 \}$.
	In fact, we have not been able to determine the precise structure of the groups listed in (2e) and (2f), nor have we been able to determine $\mathrm{cod}(G)$ for groups $G$ satisfying (2f).
  \item[$\bullet$]   An example for (2f):
  Let $S=\mathrm{SU}_3(8)$ and $N\in \mathrm{Syl}_{2}(S)$.
  Then $\mathbf{N}_{S}(N)=N \rtimes C$ where $C\cong \mathsf{C}_{63}$ and $N$ is a Suzuki $2$-group of B-type of order $2^{9}$.
  Let $P\in \mathrm{Syl}_{7}(C)$.
  Then $G:=N \rtimes P$ satisfies (2f).
  In particular, all $G$-chief factors in $N$ are isomorphic as a $P$-module over $\mathbb{F}_2$.
  \item[$\bullet$] Two examples for (2e): 
  
 (i) Let $p=2^{d}-1$ be a Mersenne prime for some odd $d$,
	let $N$ be a Suzuki $2$-group of C-type of order $2^{3d}$, and let $P\cong \mathsf{C}_{p}$ be
	a group acting transitively on the set of involutions of $N$.
	Then $N$ is an ultraspecial $2$-group such that $N/N'=U\times W$, where $U$ and $W$ are non-isomorphic faithful 
	$d$-dimensional irreducible $P$-modules (see e.g. \cite{higman1963,lewis2017}).
	So, $G:=N \rtimes P$ satisfies (2e) with elementary abelian $N/N'$.

	(ii) Let $G_0=N_0  \rtimes P$ be an example for (2f) described above,
	and let $G_1=C \rtimes P$ be a Frobenius group with complement $P$ and kernel $C\cong (\mathsf{C}_{4})^6$.
	Note that $C$ can be chosen such that every $G_0$-chief factor in $N_0$ is isomorphic, as a $P$-module over $\mathbb{F}_2$, 
	to every $G_1$-chief factor in $C$.
	Let $\widetilde{G}=\widetilde{N} \rtimes P$ be a Frobenius group with kernel $\widetilde{N}=N_0\times C$.
	Then every $\widetilde{G}$-chief factor in $\widetilde{N}$ is isomorphic as a $P$-module over $\mathbb{F}_2$.
    Let $N$ be the subdirect product of $N_0$ and $C$ obtained by identifying the $P$-modules
	$N_0/N_0'$ with $C/\Phi(C)$ (see \cite[Kapitel I, 9.11 Satz]{huppertgrouptheory}).
	Then $N$ is a $P$-invariant subgroup of $\widetilde{N}$.
	So, $G:=N \rtimes P$ satisfies (2e) with $N/N'$ having exponent $4$. 
\end{description}
	}
\end{rmkA}

The paper is organized as follows: in Section 2, we collect auxiliary results; in Section 3, we study finite solvable groups with Fitting height 2 having exactly four character codegrees; in Section 4, we
classify finite solvable groups with Fitting height 3 having exactly four character codegrees;
in Section 5, we prove Theorem \ref{thmA}.

\section{Auxiliary results}

Throughout the paper, we follow the standard conventions of \cite{huppertgrouptheory} for group theory and \cite{isaacs1994} for character theory; 
and for $n\in\mathbb N$ and $p$ a prime, we write $n_{p}$ for the largest $p$-power dividing $n$.  
For a finite group $G$, we denote by $G^{\sharp}$ the set of nontrivial elements of $G$, write $\pi(G)$ for the set of primes dividing $|G|$, and let $\exp(G)$ denote the exponent of $G$.  When $N\unlhd G$ and $\theta\in\mathrm{Irr}(N)$, we identify $\chi\in \mathrm{Irr}(G/N)$ with its inflation and view $\mathrm{Irr}(G/N)$ as a subset
of $\mathrm{Irr}(G)$,
denote by $\mathrm{Irr}(G|\theta)$ the set of irreducible characters of $G$ lying over $\theta$; and by $\mathrm{Irr}(G|N)$ we mean the complement of $\mathrm{Irr}(G/N)$ in $\mathrm{Irr}(G)$,
while $\mathrm{Irr}(G)^{\sharp}$ stands for $\mathrm{Irr}(G|G)$.  
Finally,  $\mathrm{cod}(G):=\{ \mathrm{cod}(\chi) : \chi\in \mathrm{Irr}(G) \}$ is the set
of (irreducible) character codegrees of $G$; and for $N\unlhd G$,
$\mathrm{cod}(G|N):=\{ \mathrm{cod}(\chi):\chi\in \mathrm{Irr}(G|N) \}$.
Other notation will be recalled or defined when necessary.

A nonabelian $p$-group $G$ is called \emph{special} if $G'=\mathbf{Z}(G)=\Phi(G)$.  
If, in addition, $|G'|=p$, then $G$ is said to be \emph{extraspecial}.  
When a special $p$-group $G$ further satisfies the requirement that every quotient by a maximal subgroup of its derived subgroup is extraspecial, it is termed \emph{semi-extraspecial}.
In this case, one automatically has $|G'|\le\sqrt{|G:G'|}$.  
Finally, a semi-extraspecial $p$-group $G$ whose derived subgroup attains this upper bound, i.e.\ $|G'|=\sqrt{|G:G'|}$, is called \emph{ultraspecial}.

Now, we start this section with a characterization of semi-extraspecial $p$-groups.

\begin{lem}\label{lem: semi-extra}
    Let $G$ be a $p$-group of nilpotency class $2$.
	Then 
	the following are equivalent.
	\begin{description}
		\item [(1)] $G$ is a semi-extraspecial $p$-group.
		\item [(2)] $\chi(1)=\sqrt{|G:G'|}$ for every $\chi \in \mathrm{Irr}(G|G')$.
		\item [(3)] $|\mathrm{Irr}(G|G')|=|G'|-1$.
	\end{description}
\end{lem}
\begin{proof}
	Since $G$ is of nilpotency class $2$, we have $G'\leq \mathbf{Z}(G)$.
	Hence, every character in $\mathrm{Irr}(G')$ is $G$-invariant and, 
	for distinct $\alpha,\beta\in \mathrm{Irr}(G')$,
	$\mathrm{Irr}(G|\alpha)\cap \mathrm{Irr}(G|\beta)=\varnothing$.
	Therefore, $|\mathrm{Irr}(G|G')|=|G'|-1$ 
    if and only if 
	$\lambda$ is fully ramified with respect to $G/G'$ for
	every character $\lambda\in \mathrm{Irr}(G')^\sharp$,
	or equivalently $\chi(1)=\sqrt{|G:G'|}$ for every $\chi \in \mathrm{Irr}(G|G')$,
	or equivalently $\chi$ vanishes on $G-G'$ for every $\chi \in \mathrm{Irr}(G|G')$ (see \cite[Problem 6.3 and Lemma~2.29]{isaacs1994}). 
	The result now follows from \cite[Theorems 1, 2]{lewis2017}.
\end{proof}

\subsection{Results on character codegrees}

We begin by recalling some well-known facts about character codegrees which will be employed freely in the following.

\begin{lem} \label{lem: basic facts on codegree}
	Let $G$ be a finite group and $\chi\in \mathrm{Irr}(G)$.
	\begin{description}
		\item[(1)] If $N$ is a normal subgroup of $G$ contained in $\ker(\chi)$,
			then the codegrees of $\chi$ in $G$ and in $G/N$ coincide.
		\item[(2)] If $M$ is a subnormal subgroup of $G$, then $\mathrm{cod}(\psi)\mid\mathrm{cod}(\chi)$ for every irreducible constituent $\psi$ of $\chi_M$.
		\item[(3)] If a prime $p$ divides $|G|$, then $p$ divides
			$\mathrm{cod}(\chi)$ for some $\chi\in \mathrm{Irr}(G)$.
			\item[(4)] $|G:\ker(\chi)|\leq \mathrm{cod}(\chi)^{2}$, with equality if and only if 
		$\chi=1_G$.	
	\end{description}
\end{lem}
\begin{proof}
		We refer to \cite[Lemma 2.1]{liang16} for the proofs of parts (1), (2) and (3).
       
        For part (4), observe that
$|G:\ker(\chi)|=\chi(1)\cdot\mathrm{cod}(\chi)\le\mathrm{cod}(\chi)^{2}$,
and equality forces $\chi(1)=\mathrm{cod}(\chi)=|G:\ker(\chi)|^{1/2}$, which occurs exactly when $\chi=1_{G}$.
\end{proof}

Let a finite group $A$ act via automorphisms on a finite group $G$.
We say that $A$ acts \emph{Frobeniusly} on $G$ if  
$g^{a}\neq g$ whenever $g\in G^\sharp$ and $a\in A^\sharp$ (see e.g. \cite[Page 177]{isaacsgrouptheory}).

\begin{lem}[\mbox{\cite[Theorem A]{qian2007}}]\label{lem: qww}
	Let $G$ be a finite nonabelian group of order divisible by $p$.
	If  $p$ divides
	no member in $\mathrm{cod}(G|G')$, then $P\in \mathrm{Syl}_{p}(G)$ acts Frobeniusly on $G'$.
\end{lem}

\begin{lem}[\mbox{\cite[Lemma 3.1]{alizadeh2019}}]\label{lem: |codG|=2}
	$G$ is a finite group with $|\mathrm{cod}(G)|\leq 2$
	if and only if $G$ is an elementary abelian $p$-group.
\end{lem}

Let $G$ be a finite group and $\chi\in \mathrm{Irr}(G)$.
Recall that $\mathbf{Z}(\chi):=\{ g\in G: |\chi(g)|=\chi(1) \}$ and that $\mathbf{Z}(\chi)/\ker(\chi)=\mathbf{Z}(G/\ker(\chi))$ is cyclic.

\begin{lem}\label{lem: direct product, kernel and codegrees}
	Let $G$ be the direct product of finite groups $A$ and $B$.
	Then the following hold.
	\begin{description}
		\item[(1)] For $\alpha \in \mathrm{Irr}(A)$ and $\beta \in \mathrm{Irr}(B)$, 
		we have $\ker(\alpha\times \beta)=\ker(\alpha)\times \ker(\beta)$ if and only if $|\mathbf{Z}(\alpha)/\ker(\alpha)|$ and $|\mathbf{Z}(\beta)/\ker(\beta)|$ are coprime. 
		\item[(2)] 
		For $\alpha \in \mathrm{Irr}(A)$ and $\beta \in \mathrm{Irr}(B)$,
		if $(|\mathbf{Z}(\alpha)/\ker(\alpha)|,|\mathbf{Z}(\beta)/\ker(\beta)|)=1$,
		then $\mathrm{cod}(\alpha\times \beta)=\mathrm{cod}(\alpha)\mathrm{cod}(\beta)$.
		\item[(3)] If $|A|$ and $|B|$ are coprime, then $\mathrm{cod}(G)=\{ \mathrm{cod}(\alpha)\mathrm{cod}(\beta):\alpha \in \mathrm{Irr}(A),\beta\in \mathrm{Irr}(B) \}$.
		In particular, $|\mathrm{cod}(G)|=|\mathrm{cod}(A)|\cdot |\mathrm{cod}(B)|$.
	\end{description} 
\end{lem}
\begin{proof}
  Part (1) is \cite[Problem 4.3]{isaacs1994}.  
 Part (2) follows from the equality
\[
\mathrm{cod}(\alpha\times\beta)=\frac{|A\times B:\ker(\alpha\times \beta)|}{\chi(1)}=\frac{|A:\ker(\alpha)|}{\alpha(1)}\cdot\frac{|B:\ker(\beta)|}{\beta(1)}=\mathrm{cod}(\alpha)\mathrm{cod}(\beta),
\]
where the second equality holds as $\ker(\alpha\times \beta)=\ker(\alpha)\times \ker(\beta)$ by part (1).	

For part (3), as $(|A|,|B|)=1$, every pair $\alpha \in \mathrm{Irr}(A), \beta  \in \mathrm{Irr}(B)$ satisfies the condition in part (2), so the displayed equality holds.  
Let $\alpha,\gamma\in \mathrm{Irr}(A)$ and $\beta, \delta\in \mathrm{Irr}(B)$.
Note that $(|A|,|B|)=1$,
	and hence 
	$\mathrm{cod}(\alpha)\mathrm{cod}(\beta)=\mathrm{cod}(\gamma)\mathrm{cod}(\delta)$ if and only if $\mathrm{cod}(\alpha)=\mathrm{cod}(\gamma)$ and $\mathrm{cod}(\beta)=\mathrm{cod}(\delta)$.
	Consequently, $|\mathrm{cod}(G)|=|\mathrm{cod}(A)|\cdot |\mathrm{cod}(B)|$.
\end{proof}

\begin{lem}\label{lem: kernel, abelian, cod}
  Let $G$ be a finite group with a normal subgroup
$V = V_{1} \times \dots \times V_{t}$,
where each $V_{i}$ is an abelian minimal normal subgroup of $G$ and
the $V_{i}$ are pairwise non-isomorphic as $G$-modules.
If $\lambda = \lambda_{1} \times \dots \times \lambda_{t}$ with
$\lambda_{i} \in \mathrm{Irr}(V_{i})^{\sharp}$ for every $i$, then
$\ker(\chi) \cap V = 1$ and $|V|$ divides $\mathrm{cod}(\chi)$ for every
$\chi \in \mathrm{Irr}(G | \lambda)$.
\end{lem}
\begin{proof}
Let $A\le V$ be $G$-invariant, and set $\Omega=\{V_1,\dots,V_t\}$.  
Because the $V_i$ are pairwise non-isomorphic $G$-modules, there exists a subset $\Delta$ of $\Omega$ such that $A=\bigtimes_{W\in\Delta}W$.  
Take $\chi\in\mathrm{Irr}(G|\lambda)$ with $\lambda=\lambda_1\times\dots\times\lambda_t$ and each $\lambda_i\in\mathrm{Irr}(V_i)^\sharp$.  
Then $\ker(\chi)\cap U=1$ for every $U\in\Omega$, so the $G$-invariant subgroup $\ker(\chi)\cap V$ must be trivial.  
Since $V\ker(\chi)/\ker(\chi)$ is an abelian normal subgroup of $G/\ker(\chi)$, It\^o's theorem \cite[Theorem 6.15]{isaacs1994} gives $\chi(1)\mid|G:V\ker(\chi)|$.  
Hence, $|V|$ divides $\mathrm{cod}(\chi)$ for every $\chi\in\mathrm{Irr}(G|\lambda)$.
\end{proof}

Let $G=V \rtimes H$ be a finite group where $V$ is a completely reducible $H$-module (possibly of mixed characteristic).
Each irreducible $H$-submodule of $V$ belongs to a single isomorphism class.
Let $\mathcal{S}_H(V)$ denote the \emph{set of representatives of the isomorphism classes of irreducible $H$-submodules of $V$}.
Therefore, 
$$V=\bigtimes_{W\in \mathcal{S}_H(V)} W(V),$$
where $W(V)$ denotes the \emph{$W$-homogeneous part} of $V$ (see \cite[Definition 1.12]{isaacs1994}).

For a positive integer $n$ and a set of primes $\pi$, we write $n_{\pi} = \prod_{p \in \pi} n_{p}$.

\begin{lem}\label{lem: large orbit}
	Let $G$ be a finite solvable group with a trivial Frattini subgroup.
	Then $G=V \rtimes H$ where $V=\mathbf{F}(G)$, and the following hold.
	\begin{description}
		\item[(1)] If $V$ is a Hall $\pi$-subgroup of $G$, then, for each $\mathcal{A} \subseteq \mathcal{S}_H(V)$, $\prod_{W\in \mathcal{A}}|W|=n_{\pi}$ for some $n \in \mathrm{cod}(G)$.
		\item[(2)] If $H$ is abelian, then $\prod_{W\in \mathcal{A}}|W| \in \mathrm{cod}(G)$ for each $\mathcal{A} \subseteq \mathcal{S}_H(V)$.
	\end{description}
\end{lem}
\begin{proof} 
    Since $G$ is a solvable group with a trivial Frattini subgroup, Gasch\"utz's theorem \cite[Theorem 1.12]{manzwolf1992}
	implies that $G=V \rtimes H$ where $V=\mathbf{F}(G)$ is a completely reducible $H$-module (possibly of mixed characteristic).

	(1) We work by induction on $|G|$.
    By induction, we may assume that $\mathcal{A}=\mathcal{S}_H(V)$ and $V=\bigtimes_{W\in \mathcal{A}} W$.
    Let  $\lambda=\bigtimes_{W\in \mathcal{A}} \lambda_W$ where $\lambda_W\in \mathrm{Irr}(W)^\sharp$.
	By Lemma \ref{lem: kernel, abelian, cod}, $|V|\mid \mathrm{cod}(\chi)$ for each $\chi \in \mathrm{Irr}(G|\lambda)$.
	As $V$ is a Hall $\pi$-subgroup of $G$,
	we conclude that $|V|=\mathrm{cod}(\chi)_{\pi}$ for each $\chi \in \mathrm{Irr}(G|\lambda)$.  
	
	(2)  	
	We work by induction on $|G|$.
    By induction, we may also assume that $\mathcal{A}=\mathcal{S}_H(V)$ and $V=\bigtimes_{W\in \mathcal{A}} W$.
    As $H$ is abelian, there exists a nontrivial $\lambda_W\in \mathrm{Irr}(W)$
	such that $\mathrm{I}_{H}(\lambda_W) =\mathbf{C}_{H}(W)$ by \cite[\S 19, Lemma 19.16]{huppertcharactertheory}.
    Let $\lambda=\bigtimes_{W\in \mathcal{A}} \lambda_W$.
    Then $\lambda \in \mathrm{Irr}(V)$ and 
	$$\mathrm{I}_{G}(\lambda) =\bigcap_{W\in \mathcal{A}} \mathrm{I}_{G}(\lambda_W) =\bigcap_{W\in \mathcal{A}} \mathbf{C}_{G}(W) =\mathbf{C}_{G}(V) =V.$$
	So, $\lambda^{G}\in \mathrm{Irr}(G)$.
	Note that $\ker(\lambda^{G})\cap V=1$ by Lemma \ref{lem: kernel, abelian, cod},
	and hence $\ker(\lambda^{G})=1$.
Therefore, $\prod_{W\in \mathcal{A}}|W|=|V|=\mathrm{cod}(\lambda^{G})\in \mathrm{cod}(G)$.
\end{proof}

Recall that, for a finite solvable group $G$,
$h(G)$ denotes the Fitting height of $G$.

\begin{lem}[\mbox{\cite[Corollary 1.2]{qian2025}}]\label{lem: qz}
	If $G$ is a finite solvable group,
	then $h(G)\leq |\mathrm{cod}(G)|-1$.
\end{lem}

\subsection{Results related to Frobenius groups}

In this subsection, we collect some useful results related to Frobenius groups.

\begin{lem}\label{lem: faith irred mod of Q8}
	Let $G$ be a quaternion group of order $8$.
	If $V$ is a faithful irreducible $G$-module over a finite field $\mathbb{F}$,
    then $V$ is a faithful absolutely irreducible $G$-module of dimension $2$.
	In particular, $\mathrm{End}_{\mathbb{F}[G]}(V)=\mathbb{F}$.	
\end{lem}
\begin{proof}
	Since $V$ is a faithful irreducible $G$-module over $\mathbb{F}$, 
	it follows by \cite[Chapter B, Corollary 9.4]{doerk1992}
	that $V$ is a faithful homogeneous $\mathbf{Z}(G)$-module over $\mathbb{F}$.
    Further, the cyclic group $\mathbf{Z}(G)$ has a faithful irreducible module over $\mathbb{F}$.
	So, an application of \cite[Chapter B, Theorem 9.8]{doerk1992} yields that 
	 the characteristic of $\mathbb{F}$ is not 2.
    As $\mathbb{F}$ is a splitting field for $G=\mathsf{Q}_8$ by \cite[Theorem 2.6]{huppertbook2},
	we conclude that $V$ is a faithful absolutely irreducible $G$-module of dimension $2$.
	Therefore, \cite[Kapitel V, 11.10 Hilfssatz]{huppertgrouptheory}
	yields that $\mathrm{End}_{\mathbb{F}[G]}(V)=\mathbb{F}$.	
\end{proof}

Let a finite group $H$ act coprimely on an abelian group $A$.
We say that $A$ is \emph{$H$-decomposable} if $A=B\times C$ where $B$ and $C$ 
are nontrivial $H$-invariant subgroups of $A$;
otherwise, we say that $A$ is \emph{$H$-indecomposable}.
If $A$ is $H$-indecomposable, then it is well-known that
$A$ is a homocyclic $q$-group for some prime $q$ such that
$\Omega_{i}(A)/\Omega_{i-1}(A)$
are isomorphic $H$-modules over $\mathbb{F}_q$ for $1\leq i\leq \log_q(\exp(A))$
where $\Omega_i(A):=\{ a\in A: a^{q^i}=1 \}$
(see e.g. \cite[Corollary 1]{harris1977}).

\begin{lem}\label{lem: Frobeniusly action}
	Let a finite group $H$ act via automorphisms on a finite nilpotent group $N$, and let $G=N \rtimes H$.
	Then $H$ acts Frobeniusly on $N/\Phi(N)$
    if and only if $H$ acts Frobeniusly on $N/N'$.
\end{lem}
\begin{proof}
  We proceed by induction on $|G|$.
Since $N$ is nilpotent, $N'\le\Phi(N)$ and therefore $\Phi(N/N')=\Phi(N)/N'$.
   By induction, we may assume that $N'=1$ i.e. $N$ is abelian.
   If $H$ acts Frobeniusly on $N$, then clearly $H$ acts Frobeniusly on $N/\Phi(N)$.

   Assume now that $H$ acts Frobeniusly on $N/\Phi(N)$.
	Then $H$ acts coprimely and faithfully on $N$.
   Suppose that  $N$ is $H$-decomposable, say $N=A\times B$ where $A$ and $B$ are nontrivial normal subgroups of $G$.
   Since $H$ acts Frobeniusly on $N/\Phi(N)$, it also acts Frobeniusly on $A\Phi(N)/\Phi(N)$ and therefore on $A/\Phi(A)$.
   By induction, $H$ acts Frobeniusly on $A$, and likewise on $B$.
   As a consequence, $H$ acts Frobeniusly on $N$.
   So, we may assume that $N$ is $H$-indecomposable.
   In particular, $N$ is an abelian $q$-group.
    Applying \cite[Corollary 1]{harris1977}, 
    we have that $H$ acts Frobeniusly on every $G$-chief factor in $N$.
	Consequently, $H$ acts Frobeniusly on $N$.
\end{proof}

\begin{lem}\label{lem: size of chief factor in Frob ker}
	Let a finite group $H$ act Frobeniusly on a finite group $N$, and let $G=N \rtimes H$.
	Assume that $A/B$ is a $G$-chief factor in $N$ of order $q^{m}$.
	Then the following hold.
	\begin{description}
		\item[(1)] If $H$ is cyclic, then $m$ is the multiplicative order of $q$ modulo $|H|$.
		\item[(2)] If $H$ is a quaternion group of order $8$, then $m=2$.
	\end{description}
\end{lem}
\begin{proof}
	Observe that $A/B\leq \mathbf{Z}(N/B)$ as $N$ is nilpotent,
	and hence $A/B$ is a faithful irreducible $H$-module over $\mathbb{F}_q$.
 If $H$ is cyclic, 
   then \cite[Chapter B, Theorem 9.8]{doerk1992} implies that $m$ is the multiplicative order of $q$ modulo $|H|$. 
  If $H$ is isomorphic to $\mathsf{Q}_8$,
  then, by Lemma \ref{lem: faith irred mod of Q8}, we conclude that $m=2$.
\end{proof}

\begin{lem}\label{lem: homogeneous mod and cod}
	Let $G$ be a Frobenius group with complement $H$ and elementary abelian kernel $V$.
	Suppose that $V$ is a homogeneous $H$-module over $\mathbb{F}_q$.  
	Then the following hold.
	\begin{description}
		\item[(1)] If $H$ is cyclic, then, for each $\lambda\in \mathrm{Irr}(V)^\sharp$, the cyclic $H$-module $\langle \lambda^{h}:h\in H\rangle$ is irreducible. 
		\item[(2)] Assume that $V$ is not $H$-irreducible.
		If $H$ is isomorphic to $\mathsf{Q}_8$, then there exists a $\lambda\in \mathrm{Irr}(V)^\sharp$
		such that the cyclic $H$-module $\langle \lambda^{h}:h\in H\rangle$ is not irreducible.
		Further, $\mathrm{cod}(\chi)=|V:\ker(\chi)|>q^{2}$ for each $\chi \in \mathrm{Irr}(G|\lambda)$.
	\end{description}
\end{lem}
\begin{proof}
	Set $U=\mathrm{Irr}(V)$.
	  Then $H$ acts Frobeniusly on $U$.
	  Since $V$ is a faithful homogeneous $H$-module over $\mathbb{F}_q$,
	  $U$ is also a faithful homogeneous $H$-module over $\mathbb{F}_q$ 
	  by \cite[Lemma 1]{zhang2000}.
	  In other words, $U$ is a direct sum of $t$ copies of a faithful irreducible $H$-submodule $W$.
      Set $\mathbb{E}=\mathrm{End}_{\mathbb{F}_q[H]}(W)$.
	  	Then $\mathbb{E}$ is a finite field of order $q^{e}$ by Schur's Lemma and Wedderburn's little theorem.
	  Let $\Delta$ be the set of all irreducible $H$-submodules of $U$, and set $|W|=q^{d}$.
      Then an application of \cite[Chapter B, Proposition 8.2]{doerk1992} yields that  
	$$|\Delta|=\frac{|\mathbb{E}|^t-1}{|\mathbb{E}|-1}=\frac{q^{et}-1}{q^e-1}.$$

     Assume first that $H$ is cyclic.
	 Then $e=d$ by \cite[Chapter B, Theorem 9.8]{doerk1992}.
    Note that $A\cap B=1$ for each pair of distinct $A,B\in \Delta$ and that 
   \[
      |W^\sharp| \cdot |\Delta|=(q^{d}-1)\cdot \frac{q^{dt}-1}{q^d-1}=q^{dt}-1=|U^\sharp|,
   \]
	and so every $\lambda\in U^\sharp$ lies in an irreducible $H$-submodule of $U$.
	Consequently, part (1) holds.

    Assume now that $H\cong \mathsf{Q}_8$ and that $V$ is not $H$-irreducible.
	Then, by \cite[Lemma 1]{zhang2000}, $U$ is also not $H$-irreducible i.e. $t>1$.
    Also, as $H\cong \mathsf{Q}_8$, it follows by Lemma \ref{lem: faith irred mod of Q8} that $d=2$ and $e=1$.
    Note again that $A\cap B=1$ for each pair of distinct $A,B\in \Delta$
	and that
    \[
	  |W^\sharp| \cdot |\Delta|=(q^{2}-1)\cdot \frac{q^t-1}{q-1}=(q+1)(q^{t}-1)<q^{2t}-1=|U^\sharp|
	\]
	where the inequality holds as $t>1$.
    Therefore, there exists a $\lambda\in U^\sharp$ which does not lie in any irreducible submodule of $U$.
	In other words, the cyclic $H$-module $\Lambda:=\langle \lambda^h:h\in H\rangle$ is not irreducible.
	In particular, $|\Lambda|>q^{2}$. 
    Let $\chi\in \mathrm{Irr}(G|\lambda)$.
	As $G$ is a Frobenius group with abelian kernel $V$, it follows that $\chi=\lambda^G$ and $\ker(\chi)=\bigcap_{h\in H}\ker(\lambda)^h$.
    Noting that 
	$|V/\bigcap_{h\in H} \ker(\lambda^{h})|=|\Lambda|$ by \cite[\S 5, Theorem 5.5]{huppertcharactertheory},
	we conclude that $\mathrm{cod}(\chi)=|V/\ker(\chi)|=|\Lambda|>q^{2}$.
\end{proof}

\begin{lem}\label{lem: frob, abel ker, general case, cod}
	Let $G$ be a Frobenius group with cyclic complement $H$ and abelian kernel $A\in \mathrm{Syl}_{q}(G)$.
    If all $G$-chief factors in $A$ are isomorphic as a $d$-dimensional $H$-module over $\mathbb{F}_q$,
        then $A/\ker(\lambda^{G})$ is an $H$-indecomposable abelian group of order $o(\lambda)^{d}$ for each $\lambda\in \mathrm{Irr}(A)^\sharp$. 
		Moreover, 
		$$\mathrm{cod}(G|A)=\{ q^{kd} : 1\leq k\leq \log_{q}(\exp(A)) \}.$$
\end{lem}
\begin{proof}
	We proceed by induction on $|G|$
	to show that $A/\ker(\lambda^{G})$ is an $H$-indecomposable abelian group of order $o(\lambda)^{d}$ for each $\lambda\in \mathrm{Irr}(A)^\sharp$.
	By induction, we may assume that $\ker(\lambda^G)=1$.
    As, by \cite[\S 5, Theorem 5.5]{huppertcharactertheory},
	$$|A/\ker(\lambda^G)|=|A/\bigcap_{h\in H}\ker(\lambda^h)|=|\langle \lambda^{h}:h\in H\rangle|,$$ 
	we deduce that $\mathrm{Irr}(A)=\langle \lambda^{h}:h\in H\rangle$.
	In particular, $\exp (\mathrm{Irr}(A))=\exp(A)=o(\lambda)$.
    So, \cite[\S 5, Proposition 5.8]{huppertcharactertheory} implies that $\mathrm{Irr}(A/\Phi(A))=\Omega_1(\mathrm{Irr}(A))=\langle \mu^h:h\in H\rangle$ where $\mu=\lambda^{o(\lambda)/q}$.
    Note that $G/\Phi(A)$ is a Frobenius group with cyclic complement $H\Phi(A)/\Phi(A)$ and elementary abelian kernel $A/\Phi(A)$ such that $A/\Phi(A)$ is a faithful homogeneous $H$-module over $\mathbb{F}_q$.
    In particular, $H$ acts Frobeniusly on $\mathrm{Irr}(A/\Phi(A))=\Omega_1(\mathrm{Irr}(A))$.
	Applying \cite[Lemma 1]{zhang2000},
	we also deduce that $\Omega_1(\mathrm{Irr}(A))=\mathrm{Irr}(A/\Phi(A))$ is a homogeneous $H$-module over $\mathbb{F}_q$.
	Thus, part (1) of Lemma \ref{lem: homogeneous mod and cod} implies that $\Omega_1(\mathrm{Irr}(A))$ is an irreducible $H$-module over $\mathbb{F}_q$ with dimension $d$.
	So, an application of \cite[Theorem]{harris1977} yields that $\mathrm{Irr}(A)$ is $H$-indecomposable.
    If $A=B\times C$ where $B$ and $C$ are $H$-invariant,
	then $\mathrm{Irr}(A)=\mathrm{Irr}(B)\times \mathrm{Irr}(C)$ where $\mathrm{Irr}(B)$ and $\mathrm{Irr}(C)$ 
	are also $H$-invariant.
	Thus, we deduce that $A$ is $H$-indecomposable.
    Recalling that $\exp(A)=o(\lambda)$, we conclude that $A$ is an $H$-indecomposable abelian group of order $o(\lambda)^{d}$.

    For $\lambda\in \mathrm{Irr}(A)^\sharp$, since $\lambda^{G}\in \mathrm{Irr}(G)$ has degree $|H|$,
	we have that $\mathrm{cod}(\lambda^{G})=|A/\ker(\lambda^{G})|=o(\lambda)^{d}$ where $q\leq o(\lambda)\leq \exp(A)$.
    Note that, for $1\leq k\leq \log_q(\exp(A))$, there exists a $\lambda\in \mathrm{Irr}(A)^\sharp$ such that $o(\lambda)=q^{k}$.
    Consequently, $\mathrm{cod}(G|A)=\{ q^{kd} : 1\leq k\leq \log_{q}(\exp(A)) \}$.
\end{proof}

\begin{lem}\label{lem: P act fpf on N/N'}
	Let $G=N \rtimes P$ where $N\in \mathrm{Syl}_{q}(G)$ and $P\cong \mathsf{C}_{p}$.
	Assume that $P$ acts Frobeniusly on $N/N'$.
	Then the following hold.
	\begin{description}
		\item[(1)] $G=\mathbf{O}^{p'}(G)$. In particular, $N=\mathbf{F}(G)$ is the unique maximal normal subgroup of $G$.
		\item[(2)] Assume that $\log_q(|N'|)$ is smaller than the multiplicative order of $q$ modulo $p$. 
		Then $N'=\mathbf{C}_{N}(P)$. 
		If, in addition $N'\leq \mathbf{Z}(N)$, then $N'=\mathbf{Z}(G)$. 
	\end{description}
\end{lem}
\begin{proof}
	As $G/N'$ is a Frobenius group with kernel $N/N'\in \mathrm{Syl}_{q}(G/N')$ and complement $PN'/N'\in \mathrm{Syl}_{p}(G/N')$,
	it follows that $\mathbf{O}^{p'}(G/N')=G/N'$.
	Note that $\mathbf{O}^{p'}(G)N'/N'\geq \mathbf{O}^{p'}(G/N')$,
	and so $G=\mathbf{O}^{p'}(G)N'$.
	Since $N$ is a normal Sylow $q$-subgroup of $G$, $N'\leq \Phi(N)\leq \Phi(G)$.
	Consequently, $G=\mathbf{O}^{p'}(G)$.

	Assume that $\log_q(|N'|)$ is smaller than the multiplicative order of $q$ modulo $p$. 
    Then every $G$-chief factor in $N'$ is centralized by $P$ by part (1) of Lemma \ref{lem: size of chief factor in Frob ker}.
	Since $P$ acts Frobeniusly on $N/N'$,
	we have that $N'=\mathbf{C}_{N}(P)$.
	If, in addition $N'\leq \mathbf{Z}(N)$, then $N'=\mathbf{Z}(G)$.
\end{proof}

\begin{lem}\label{lem: dim of faithful module}
	Let $p,q,r$ be primes such that $p\neq q$ and $q\neq r$, 
	and 
	$G$ a Frobenius group with complement $P\cong \mathsf{C}_{p}$ and kernel $Q\cong \mathsf{C}_{q}$.
	If $V$ is a faithful irreducible $G$-module over $\mathbb{F}_r$, then one of the following holds.
	\begin{description}
		\item[(1)] $V$ is $Q$-irreducible, and $\dim_{\mathbb{F}_r}(V)$ is the multiplicative order of $r$ modulo $q$.
		\item[(2)] $V$ is not $Q$-irreducible, and $\dim_{\mathbb{F}_r}(V)$ equals $p$ times the multiplicative order of $r$ modulo $q$. 
	\end{description}
\end{lem}
\begin{proof}
	Let $V$ be a faithful irreducible $G$-module over $\mathbb{F}_r$.
	Since $|G/Q|=p$ is a prime and $\mathbf{C}_{G}(Q)=Q$,
	\cite[Theorem 0.1, Lemma 2.2]{manzwolf1992}
	implies that either $V$ is a faithful irreducible $Q$-module
	over $\mathbb{F}_r$, or $V=V_1\oplus  \cdots \oplus  V_p$ where $V_i$ are non-isomorphic irreducible $Q$-module
	over $\mathbb{F}_r$.

	If the former holds, as $V$ is a faithful irreducible $Q$-module over $\mathbb{F}_r$ where $Q\cong \mathsf{C}_{q}$ and $q\neq r$,
	it follows by Lemma \ref{lem: size of chief factor in Frob ker} that $\dim_{\mathbb{F}_r}(V)$ is the multiplicative order of $r$ modulo $q$.
    Assume that the latter holds.
	Then all $V_i$ are faithful irreducible $Q$-modules over $\mathbb{F}_r$.
	In fact, otherwise $V_i$ are isomorphic to the trivial irreducible $Q$-module, a contradiction.
	Since $Q$ is a cyclic group of order $q$ such that $q\neq r$, $\dim_{\mathbb{F}_r}(V_i)$ is the multiplicative order of $r$ modulo $q$.
    As $\dim_{\mathbb{F}_r}(V)=p\cdot \dim_{\mathbb{F}_r}(V_i)$, part (2) holds.
\end{proof}

\begin{lem}\label{lem: 2-Frobenius}
	Let $p,q,r$ be primes, and $G=V \rtimes H$.
	Assume that  
	$H$ is a Frobenius group with complement $P\cong \mathsf{C}_{p}$ and kernel $Q\cong \mathsf{C}_{q}$,
	and that $VQ$ is a Frobenius group with complement $Q$ and kernel $V\cong (\mathsf{C}_{r})^{pm}$.
    If $q=\frac{r^{pm}-1}{r^{m}-1}$,
	then $V$ is minimal normal in $VQ$.	
\end{lem}
\begin{proof}
	Assume that $V$ is not minimal normal in $VQ$.
	Then $V$ is not $Q$-irreducible.
	By Lemma \ref{lem: dim of faithful module}, $m$ is the multiplicative order of $r$ modulo $q$.
	In particular, $q\mid r^{m}-1$.
	Therefore,
	\[
	 \frac{r^{pm}-1}{r^{m}-1}=r^{(p-1)m}+\cdots +r^{m}+1 \equiv p \nequiv 0~(\mathrm{mod}~q)
	\]
	which contradicts $q=\frac{r^{pm}-1}{r^{m}-1}$.
\end{proof}

\section{Solvable groups with Fitting height 2}

Let $G$ be a finite solvable group with Fitting height 2 and $N$ the nilpotent residual of $G$.
Then $N$ is contained in $\mathbf{F}(G)$.
Assume that 
$G/N$ is a $p$-group.
Then $(|N|,|G/N|)=1$.
So, the Schur-Zassenhaus theorem implies that 
$G=N \rtimes P$ where $P\in \mathrm{Syl}_{p}(G)$.

For a finite nilpotent group $G$, recall that $c(G)$ denotes the nilpotency class of $G$.

\begin{lem}\label{lem: |cod(G)|<=4}
	Let $G=N \rtimes P$ be a finite solvable group with Fitting height $2$ where $N$ is the nilpotent residual of $G$ and $P\in \mathrm{Syl}_{p}(G)$.
	Assume that $|\mathrm{cod}(G)|\leq 4$.
	Then the following hold.
	\begin{description}
		\item[(1)] $c(P)\leq 2$.
		\item[(2)] $N$ is a $q$-group contained in $\mathbf{F}(G)$.
	\end{description}
\end{lem}
\begin{proof}
    As $G$ has Fitting height $2$,
	it forces that the nilpotent residual $N$ of $G$ is contained in $\mathbf{F}(G)$.

	(1) Let $N/E$ be a $G$-chief factor.
	Applying Lemma \ref{lem: kernel, abelian, cod} to $G/E$,
    we deduce that $|N/E|\mid n$ for some $n\in \mathrm{cod}(G)$.
	Note that $|\mathrm{cod}(G)|\leq 4$ and that $n\notin \mathrm{cod}(G/N)=\mathrm{cod}(P)$, and hence $|\mathrm{cod}(P)|\leq 3$.
	Therefore, $c(P)\leq 2$ by \cite[Theorem 1.2]{du2016}.

	(2) Let $G$ be a counterexample of minimal possible order.
    By the minimality of $G$, we have $\mathbf{C}_{P}(N)=\Phi(N)=1$.
	So, the nilpotent group $N=\mathbf{F}(G)$ is the socle of $G$ and $\Phi(G)=1$.
	Since $|\mathrm{cod}(G)|\leq 4$ and $1,p\in \mathrm{cod}(G)$,
	part (1) of Lemma \ref{lem: large orbit} forces $N$ to be a $q$-group, a contradiction.
\end{proof}

\begin{hy}\label{hy: height=2, |cod(G/N)|=3}
   Let $G=N \rtimes P$ be a finite solvable group with Fitting height $2$ where $N$ is the nilpotent residual of $G$ and $P\in \mathrm{Syl}_{p}(G)$.
   Assume that $|\mathrm{cod}(G/N)|=3$.
\end{hy}

Assume Hypothesis \ref{hy: height=2, |cod(G/N)|=3} and that $|\mathrm{cod}(G)|=4$.
Then $P$ has nilpotency class at most 2 and $N$ is a $q$-group by Lemma \ref{lem: |cod(G)|<=4}.

\begin{lem}\label{lem: height=2, |cod(G/N)|=3, forinduction}
   Assume Hypothesis \ref{hy: height=2, |cod(G/N)|=3}
   and that $|\mathrm{cod}(G)|=4$. 
   If $D$ is a $G$-invariant proper subgroup of $N$,
	 then $|\mathrm{cod}(G/D)|=4$.
\end{lem}
\begin{proof}
    Let $N/E$ be a $G$-chief factor such that $D\leq E$.
   Then $|N/E|\mid n$ for some $n\in \mathrm{cod}(G/E)$ by Lemma \ref{lem: kernel, abelian, cod}. 
   As $\mathrm{cod}(G/N)=\{ 1,p,p^{k} \}\subseteq \mathrm{cod}(G/E)$ where $k>1$,
   we deduce that $|\mathrm{cod}(G/E)|\geq 4$.
   Noting that $\mathrm{cod}(G/E)\subseteq \mathrm{cod}(G/D)\subseteq \mathrm{cod}(G)$,
   we conclude that $|\mathrm{cod}(G/D)|=4$.
\end{proof}

Let a finite group $P$ act via automorphism on a finite group $V$.
Then $P$ acts as a permutation group on $V^\sharp$.
We say this action is \emph{$\frac{1}{2}$-transitive} if every orbit shares the same size.

\begin{prop}\label{prop: height=2, |cod(G/N)|=3, N unique minimal normal}
	Assume Hypothesis \ref{hy: height=2, |cod(G/N)|=3} and that $N$ is the unique minimal normal subgroup of $G$.
	If $|\mathrm{cod}(G)|=4$, then one of the following holds.
	\begin{description}
		\item[(1)] $G$ is a Frobenius group with complement $P$ and kernel $N$ such that one of the following holds.
		 \begin{description}
			\item[(1a)] $P\cong \mathsf{C}_{p^{2}}$, $N\cong (\mathsf{C}_{q})^{d}$ where
			$d$ is the multiplicative order of $q$ modulo $p^{2}$,
			and $\mathrm{cod}(G)=\{ 1,p,p^{2},q^{d} \}$.
			\item[(1b)] $P\cong \mathsf{Q}_8$, $N\cong (\mathsf{C}_{q})^{2}$, and $\mathrm{cod}(G)=\{ 1,2,4,q^{2} \}$.
		 \end{description}
		 \item[(2)]
		 $\mathrm{I}_{P}(\lambda)$ is a non-normal subgroup of order $2$ of $P$ for each $\lambda \in \mathrm{Irr}(N)^\sharp$,
		 and one of the following holds.
		 \begin{description}
			\item[(2a)] $P\cong \mathsf{D}_{8}$, $N\cong (\mathsf{C}_{3})^{2}$, and $\mathrm{cod}(G)=\{ 1,2,4, 18 \}$.
			\item[(2b)] $P\cong\mathsf{SmallGroup}(16,13)$, $N\cong (\mathsf{C}_{5})^{2}$, and $\mathrm{cod}(G)=\{ 1,2,8,50 \}$.
			\item[(2c)] $P\cong \mathsf{ES}(2^{5}_{-})$, $N\cong (\mathsf{C}_{3})^{4}$, and $\mathrm{cod}(G)=\{ 1,2,8,162 \}$.
		 \end{description}
	\end{description}
\end{prop}
\begin{proof}
  Since $N$ is the unique minimal normal subgroup of the solvable group $G$, $N=\mathbf{F}(G)$ is an elementary abelian $q$-group.
  In particular, $\Phi(G)=1$.

  Let $\lambda$ be a nontrivial character in $\mathrm{Irr}(N)$ and set $T=\mathrm{I}_{G}(\lambda)$.
  Note that $(|T/N|,|N|)=1$,
  and hence $\lambda$ extends to some $\hat{\lambda}\in \mathrm{Irr}(T)$.
  So, Clifford's correspondence yields that $\hat{\lambda}^{G}\in \mathrm{Irr}(G)$.
  Since $N=\mathbf{F}(G)$ is minimal normal in $G$,
  it follows by Lemma \ref{lem: kernel, abelian, cod} that $\ker(\hat{\lambda}^{G})=1$.
  So,
  \[
   \mathrm{cod}(\hat{\lambda}^{G})=\frac{|G|}{\hat{\lambda}^{G}(1)}=|T|=|N|\cdot |\mathrm{I}_{P}(\lambda)|.
  \]
  Note that $\mathrm{cod}(\hat{\lambda}^{G})\notin \mathrm{cod}(G/N)$ and that $|\mathrm{cod}(G)|=4$,
  and so $|\mathrm{I}_{P}(\lambda)|$ is a constant for each $\lambda\in \mathrm{Irr}(N)^\sharp$.
  Set $V=\mathrm{Irr}(N)$.
  Hence, $P$ acts $\frac{1}{2}$-transitively on $V^\sharp$.
  Note that $\mathbf{C}_{P}(V)=\mathbf{C}_{P}(N)=1$ by \cite[Lemma 1]{zhang2000} and that $c(P)\leq 2$ by part (1) of Lemma \ref{lem: |cod(G)|<=4}.
  Applying \cite[Theorem II]{isaacspassman1966}, we conclude that either $G$ is a Frobenius group with complement $P$ and kernel $N$, 
  or $|\mathrm{I}_{P}(\lambda)|=2$ for each $\lambda \in V^\sharp$.
  Assume that the latter holds. 
  Checking the groups listed in (i)-(iii) of \cite[Theorem II]{isaacspassman1966} case by case and applying $\mathsf{GAP}$ \cite{gap},
  we conclude that part (2) holds.
  Assume that the former holds.
  Then the Frobenius complement $P$ is either cyclic or generalized quaternion.
  As $c(P)\leq 2$ and $|\mathrm{cod}(P)|=|\mathrm{cod}(G/N)|=3$,
  the Sylow $p$-subgroup $P$ is isomorphic to either $\mathsf{C}_{p^{2}}$ or $\mathsf{Q}_8$.
  If $P$ is isomorphic to $\mathsf{C}_{p^{2}}$,
  then (1a) follows from \cite[Chapter B, Theorem 9.8]{doerk1992}, Lemma \ref{lem: large orbit} and a direct computation. 
  If $P$ is isomorphic to $\mathsf{Q}_8$,
  then (1b) is obtained via Lemmas \ref{lem: faith irred mod of Q8}, \ref{lem: large orbit} and the corresponding calculation.
\end{proof}

\begin{lem}\label{lem: height=2, |cod(G/N)|=3, C_P(N)=1}
	Assume Hypothesis \ref{hy: height=2, |cod(G/N)|=3} and that $|\mathrm{cod}(G)|=4$.
	If $D$ is a $G$-invariant proper subgroup of $N$,
	then $\mathbf{C}_{PD/D}(N/D)=1$.
	In particular, $N/D=\mathbf{F}(G/D)$ is a $q$-group.
\end{lem}
\begin{proof}
	Let $G$ be a counterexample of minimal possible order.
	Let $N/E$ be a $G$-chief factor such that $D\leq E$.
	Note that 
		 $$\frac{\mathbf{C}_{PD/D}(N/D)E/D}{E/D}$$
	is isomorphic to a subgroup of $\mathbf{C}_{PE/E}(N/E)$,
	and so $\mathbf{C}_{PE/E}(N/E)=1$ implies that $\mathbf{C}_{PD/D}(N/D)=1$.
    Since, by Lemma \ref{lem: height=2, |cod(G/N)|=3, forinduction}, $G/E$ satisfies the hypothesis of this lemma,
	we deduce that $D=E=1$ by the minimality of $G$.
    In other words, $N$ is minimal normal in the solvable group $G$.

	Set $C=\mathbf{C}_{P}(N)$. We claim now that $\mathrm{cod}(P/C)=\mathrm{cod}(P)$.
	Otherwise $C>1$ and $|\mathrm{cod}(P/C)|<3$.
	 So, by Lemma \ref{lem: |codG|=2}, $P/C$ is an elementary abelian $p$-group.
	 As $N$ is a faithful irreducible $P/C$-module,
	 $G/C$ is a Frobenius group with complement $P/C\cong \mathsf{C}_{p}$.
	Write $|N|=q^d$.
    Hence, it is routine to check that $\mathrm{cod}(G/C)=\{ 1,p,q^{d} \}$.
	Recall that $|\mathrm{cod}(P)|=3$,
	and so $\mathrm{cod}(G/N)=\mathrm{cod}(P)=\{ 1,p,p^{k} \}$ for some $k>1$.
    Thus, $\mathrm{cod}(G)=\{ 1,p,p^{k},q^{d} \}$.
    Let $\lambda\in \mathrm{Irr}(N)^\sharp$, $\mu\in \mathrm{Irr}(C/C')^\sharp$ and $\chi \in \mathrm{Irr}(G|\lambda\times \mu)$.
	Note that $\mathrm{I}_{G}(\lambda\times \mu)=\mathrm{I}_{G}(\lambda)\cap \mathrm{I}_{G}(\mu)=NC$,
	and hence $\chi=(\lambda\times \mu)^G$. 
	In particular, $\chi(1)=p$.
	Also, 
	$$\ker(\chi)=\ker((\lambda\times \mu)^{G})=\bigcap_{g\in G} \ker(\lambda^g\times \mu^g)=\bigcap_{g\in G} (\ker(\lambda^g)\times \ker(\mu^g))=\ker(\lambda^{G})\times \ker(\mu^{G})=\ker(\mu^{G})$$
	where the third equality holds by part (1) of Lemma \ref{lem: direct product, kernel and codegrees}.
    So, 
	\[
	\mathrm{cod}(\chi)=\frac{|G:\ker(\chi)|}{\chi(1)}=\frac{|G:\ker(\mu^{G})|}{p}=q^{d}\cdot 
	|C:\ker(\mu^{G})|.
	\]
    Since $\ker(\mu^{G})\leq \ker(\mu)<C$, $\mathrm{cod}(\chi)\notin \mathrm{cod}(G)$, a contradiction.
    Therefore, $\mathrm{cod}(P/C)=\mathrm{cod}(P)$.
	
	Therefore, $\mathrm{cod}(G/C)=\mathrm{cod}(G)$.
	It follows by the minimality of $G$ that $C$ is minimal normal in $G$.
    In particular, $C\cong \mathsf{C}_{p}$ is central in $G$.
	Let $\lambda \in \mathrm{Irr}(N)^\sharp$, $\mu\in \mathrm{Irr}(C)^\sharp$ and $\chi \in \mathrm{Irr}(G|\lambda\times \mu)$.
	Then $\mathrm{I}_{G}(\lambda\times \mu)=\mathrm{I}_{G}(\lambda)=\mathrm{I}_{G}(\lambda\times 1_C)$, and $\ker(\chi)\cap N=\ker(\chi)\cap C=1$.
	Note that $\mathbf{F}(G)=NC$ and that $(|N|,|C|)=1$, and so $\ker(\chi)=1$.
	Since $G/C$ satisfies the hypothesis of Proposition \ref{prop: height=2, |cod(G/N)|=3, N unique minimal normal},
	an application of Proposition \ref{prop: height=2, |cod(G/N)|=3, N unique minimal normal} yields that
	either $\mathrm{I}_{G}(\lambda\times \mu)=NC$ or $|\mathrm{I}_{G}(\lambda\times \mu):NC|=2=p$.
    In particular, $\lambda\times \mu$ extends to $\mathrm{I}_{G}(\lambda\times \mu)$.
	Applying Gallagher's theorem \cite[Corollary 6.17]{isaacs1994}, we deduce that $\chi=\theta^G$ where $\theta$ is an extension of $\lambda\times \mu$ in $\mathrm{I}_{G}(\lambda\times \mu)$.
	Therefore,
	\[
	\mathrm{cod}(\chi)=\frac{|G|}{\chi(1)}=\frac{|G|}{|G:\mathrm{I}_{G}(\lambda\times \mu)|}=q^{d}p\cdot |\mathrm{I}_{G}(\lambda\times \mu):NC|.
	\]
   Observe again that $G/C$ satisfies the hypothesis of Proposition \ref{prop: height=2, |cod(G/N)|=3, N unique minimal normal},
   and so an application of Proposition \ref{prop: height=2, |cod(G/N)|=3, N unique minimal normal}
   yields the final contradiction that $\mathrm{cod}(\chi)\notin \mathrm{cod}(G)$.	
\end{proof}

\begin{lem}\label{lem: height=2, |cod(G/N)|=3, N is homogeneous}
   Assume Hypothesis \ref{hy: height=2, |cod(G/N)|=3} and that $|\mathrm{cod}(G)|=4$.
    Then $N$ is a faithful homogeneous $P$-module over $\mathbb{F}_q$.  
\end{lem}
\begin{proof}
	By Lemma \ref{lem: height=2, |cod(G/N)|=3, C_P(N)=1}, $N=\mathbf{F}(G)$ is a $q$-group.
	Let $G$ be a counterexample of minimal possible order.
	Since $1,p,p^{k}\in\mathrm{cod}(G/N)\subseteq\mathrm{cod}(G)$ with $k>1$, the assumption $\Phi(N)=1$ would force $N$ to be a homogeneous $P$-module over $\mathbb{F}_{q}$ by part (1) of Lemma~\ref{lem: large orbit}, contradicting our choice of $G$.  
	Hence $\Phi(N)>1$.
    Let $E$ be a minimal $G$-invariant subgroup of $N$.
	Since $|\mathrm{cod}(G/E)|=4$ by Lemma \ref{lem: height=2, |cod(G/N)|=3, forinduction},
	minimality of $G$ yields that $N/E$ is a homogeneous $P$-module over $\mathbb{F}_{q}$.  
	Consequently, $E=\Phi(N)$ is the unique minimal normal subgroup of $G$ and $E\le\mathbf{Z}(N)$. 
	Also, part (1) of Lemma \ref{lem: large orbit} yields that $\mathrm{cod}(G/E|N/E)=\{ n \}$ with $n_q$ the order of a $G$-chief factor in $N/E$. 
    Thus, $\mathrm{cod}(G)=\{ 1,p,p^{k},n \}$ where $k>1$.
    
	Let $\theta\in \mathrm{Irr}(N|E)$ and set $T=\mathrm{I}_{G}(\theta)$.
	Since $(|N|,|T/N|)=1$, $\theta$ extends to $\hat{\theta}\in \mathrm{Irr}(T)$.
	So, Clifford's correspondence forces $\chi:=\hat{\theta}^{G}\in \mathrm{Irr}(G)$.
	As $\ker(\chi)$ has trivial intersection with the unique minimal normal subgroup $E$ of $G$, $\ker(\chi)=1$.
	 Therefore,
\begin{equation}\label{chi}
		\mathrm{cod}(\chi)=\frac{|G|}{\chi(1)}=|T:N|\cdot \frac{|N|}{\theta(1)}=n.
\end{equation}
    In particular, $|N|=n_q\cdot \theta(1)$.
     If $N$ is abelian, then $\mathrm{cod}(\chi)_q=|N/E|\cdot |E|>n_q$, a contradiction. 
     So, $N$ is of nilpotency class 2.
	 Set $n_q=q^{d}$ and $|E|=q^e$.
	 As $N/E$ is a homogeneous $G$-module over $\mathbb{F}_q$, 
	 $|N:E|=q^{(a+1)d}$ for some nonnegative integer $a$, and so $\theta(1)=q^{ad+e}$.
     As $\theta(1)^{2}\leq |N:\mathbf{Z}(N)|\leq |N:E|=q^{(a+1)d}$,
     we obtain $2(ad+e)\leq (a+1)d$, whence $a=0$, $2e\leq d$, $\theta(1)=q^{e}$ and $N/E$ is a $G$-chief factor.
     Therefore, $N$ is a special $q$-group.
      Lemma \ref{lem: height=2, |cod(G/N)|=3, C_P(N)=1} also
     forces $N/E$ to be the unique minimal normal subgroup of $G/E$.
	 Hence, Proposition \ref{prop: height=2, |cod(G/N)|=3, N unique minimal normal} shows that either $G/E$ is one of the Frobenius groups listed in part (1) of Proposition \ref{prop: height=2, |cod(G/N)|=3, N unique minimal normal},
	 or $G/E$ is one of the groups listed in part (2) of Proposition \ref{prop: height=2, |cod(G/N)|=3, N unique minimal normal}.
     
	 Assume that $G/E$ is one of the Frobenius groups listed in part (1) of Proposition \ref{prop: height=2, |cod(G/N)|=3, N unique minimal normal}.
	 Then $\mathrm{cod}(\chi)=n=n_q$.
	 So, by (\ref{chi}),
	 $\mathrm{I}_{G}(\theta)=T=N$ for each $\theta\in \mathrm{Irr}(N|E)$.
	 Therefore, $G$ is a Frobenius group with complement $P\in \{ \mathsf{C}_{p^{2}},\mathsf{Q}_8 \}$ and kernel $N$.
	 Note that every faithful irreducible module of $P$ over $\mathbb{F}_q$ has the same dimension.
	 As a consequence, $q^{e}=|E|=|N/E|=q^{d}$ which contradicts $2e\leq d$.
   
    Assume that $G/E$ is one of the groups listed in part (2) of Proposition \ref{prop: height=2, |cod(G/N)|=3, N unique minimal normal}.
	Then $p=2$, $q\in \{ 3,5 \}$, $\mathrm{cod}(\chi)_2=n_2=2$ and $\mathrm{cod}(\chi)_q=|N/E|=q^{d}$. 
	We also claim that $\theta(1)^{2}=|N/E|$.
    In fact, by part (2) of Proposition \ref{prop: height=2, |cod(G/N)|=3, N unique minimal normal},
	$|N/E|\in \{ 3^{2}, 5^{2}, 3^{4} \}$;
    if $|N/E|\in \{ 3^{2},5^{2} \}$, as $\theta(1)=|E|=q^{e}$ where $2e\leq d=2$,
	we deduce that $\theta(1)^{2}=q^{2}=|N/E|$;
	if $|N/E|=3^{4}$, as $2e\leq d=4$, we have that either $e=1$ or $2$;
	if the former holds, note that $N$ is an extraspecial $3$-group of order $3^{5}$,
	and so $\theta(1)=3^{2}>3^e$, a contradiction;
	so, $e=2$ and $\theta(1)^{2}=3^{2e}=3^{4}=|N/E|$.
    Hence, there exists a $\lambda\in \mathrm{Irr}(E)^\sharp$
	such that $\theta=\frac{1}{\theta(1)}\lambda^{N}$ (see e.g. \cite[Problem 6.3]{isaacs1994}).
    In particular, $\mathrm{I}_{P}(\theta)=\mathrm{I}_{P}(\lambda)$.
    Observing that  
	$|\mathrm{I}_{P}(\lambda)|=|\mathrm{I}_{G}(\theta):N|=\mathrm{cod}(\chi)_2=2$ by (\ref{chi}),
    we conclude a contradiction that $|P:\mathrm{I}_{P}(\lambda)|>|\mathrm{Irr}(E)|$.
    In fact, if $G/E$ satisfies (2a) of Proposition \ref{prop: height=2, |cod(G/N)|=3, N unique minimal normal}, as in this case $|P|=8$ and $|E|=3$,
	we conclude a contradiction that $|P:\mathrm{I}_{P}(\lambda)|=4>|\mathrm{Irr}(E)|=3$;
	if $G/E$ satisfies (2b) of Proposition \ref{prop: height=2, |cod(G/N)|=3, N unique minimal normal}, as in this case $|P|=16$ and $|E|=5$,
	we conclude a contradiction that $|P:\mathrm{I}_{P}(\lambda)|=8>|\mathrm{Irr}(E)|=5$;
	if $G/E$ satisfies (2c) of Proposition \ref{prop: height=2, |cod(G/N)|=3, N unique minimal normal}, as in this case $|P|=32$ and $|E|=9$,
	we conclude a contradiction that $|P:\mathrm{I}_{P}(\lambda)|=16>|\mathrm{Irr}(E)|=9$.
\end{proof}

\begin{thm}\label{thm: height=2, |cod(G/N)|=3, classification}
	Assume Hypothesis \ref{hy: height=2, |cod(G/N)|=3}.
	Then $|\mathrm{cod}(G)|=4$ if and only if one of the following holds.
	\begin{description}
		\item[(1)] $G$ is a Frobenius group with complement $P$ and kernel $N$ such that one of the following holds.
		 \begin{description}
			\item[(1a)] $P\cong \mathsf{C}_{p^{2}}$, $N\cong (\mathsf{C}_{q})^{td}$ is a homogeneous $P$-module where
			$d$ is the multiplicative order of $q$ modulo $p^{2}$.
			\item[(1b)] $P\cong \mathsf{Q}_8$ and $N\cong (\mathsf{C}_{q})^{2}$.
		 \end{description}
		 \item[(2)] $N=\mathbf{F}(G)$, 
		 $\mathbf{C}_{P}(x)$ is a non-normal subgroup of order $2$ of $P$ for each $x\in N^\sharp$,
		 and one of the following holds.
		 \begin{description}
			\item[(2a)] $P\cong \mathsf{D}_{8}$ and $N\cong (\mathsf{C}_{3})^{2}$.
			\item[(2b)] $P\cong\mathsf{SmallGroup}(16,13)$ and $N\cong (\mathsf{C}_{5})^{2}$. 
			\item[(2c)] $P\cong \mathsf{ES}(2^{5}_{-})$ and $N\cong (\mathsf{C}_{3})^{4}$.
		 \end{description}
	\end{description}
\end{thm}
\begin{proof}
	We assume first that $|\mathrm{cod}(G)|=4$.
	By Lemma \ref{lem: height=2, |cod(G/N)|=3, N is homogeneous},
	$N$ is a faithful homogeneous $P$-module over $\mathbb{F}_q$.
    Let $E$ be a minimal normal subgroup of $G$ in $N$ and $F$ a $G$-invariant complement of $E$ in $N$.
	Note that $EP\cong G/F$,
	and so, by Lemmas \ref{lem: height=2, |cod(G/N)|=3, forinduction} and \ref{lem: height=2, |cod(G/N)|=3, C_P(N)=1},
	 $EP$ satisfies the hypothesis of Proposition  \ref{prop: height=2, |cod(G/N)|=3, N unique minimal normal}.
	So,
	Proposition \ref{prop: height=2, |cod(G/N)|=3, N unique minimal normal}
    forces that either $EP$ is one of the Frobenius groups listed in part (1) of Proposition \ref{prop: height=2, |cod(G/N)|=3, N unique minimal normal} or $EP$ is one of the groups listed in part (2) of Proposition \ref{prop: height=2, |cod(G/N)|=3, N unique minimal normal}.
	If the former holds and $P\cong \mathsf{C}_{p^{2}}$, as $N$ is a faithful homogeneous $P$-module over $\mathbb{F}_q$, we conclude that (1a) holds.
	If the former holds and $P\cong \mathsf{Q}_8$, 
	then Proposition \ref{prop: height=2, |cod(G/N)|=3, N unique minimal normal} 
	yields that 
	$\mathrm{cod}(G)=\mathrm{cod}(EP)=\{ 1,2,4, q^{2} \}$,
	and part (2) of Lemma \ref{lem: homogeneous mod and cod} 
    then forces (1b).

	Assume now that the latter holds.
	Recall that $|\mathrm{cod}(G)|=4$,
	and so 
	$$\mathrm{cod}(G)=\mathrm{cod}(G/F)=\mathrm{cod}(EP)=\{ 1,2,2^{a},2|E| \}$$
	 where $a>1$ and $|E|\geq q^{2}$ by Proposition \ref{prop: height=2, |cod(G/N)|=3, N unique minimal normal}.
	Hence, it remains to show that $N$ is minimal normal in $G$.
    Let $\lambda\in V^\sharp$ where $V:=\mathrm{Irr}(N)$, and set $T=\mathrm{I}_{G}(\lambda)$.
	Since $(|T/N|,|N|)=1$,
	$\lambda$ extends to $\hat{\lambda}\in \mathrm{Irr}(T)$.
	So, Clifford's correspondence yields $\hat{\lambda}^{G}\in \mathrm{Irr}(G)$.
	Set $K=\ker(\hat{\lambda}^{G})$ and $D=K\cap N$.
	As $D<N$, $G/D$ satisfies Hypothesis \ref{hy: height=2, |cod(G/N)|=3} and, by Lemma \ref{lem: height=2, |cod(G/N)|=3, forinduction}, 
	$|\mathrm{cod}(G/D)|=4$.
    So, $N/D=\mathbf{F}(G/D)$ by Lemma \ref{lem: height=2, |cod(G/N)|=3, C_P(N)=1}.
	Since $K/D\cap N/D=1$, we have $K=D< N$.
    Therefore,
	\[
	 \mathrm{cod}(\hat{\lambda}^{G})=\frac{|G:K|}{\hat{\lambda}^{G}(1)}=\frac{|G:K|}{|G:T|}=|N:K|\cdot |T:N|,
	\]
    where $|N:K|=q^{k}$ for some $k>0$.
	So, $\mathrm{cod}(\hat{\lambda}^{G})=2|E|$,
	and therefore
	$|\mathrm{I}_{P}(\lambda)|=|T:N|=2$ for each $\lambda\in V^\sharp$.
    Equivalently, $P$ acts $\frac{1}{2}$-transitively on $V^\sharp$.
	By \cite[Theorem I]{isaacspassman1966}, $V$ is $P$-irreducible,
	so $N$ is minimal normal in $G$ by \cite[Lemma 1]{zhang2000}.
	Finally, one uses $\mathsf{GAP}$ \cite{gap} to verify 
	that $\mathbf{C}_{P}(x)$ is a non-normal subgroup of order $2$ of $P$ for each $x\in N^\sharp$.

	Conversely, we assume that either part (1) or part (2) holds.
	If part (1) holds,
    then $\{  1,p,p^{2} \} = \mathrm{cod}(G/N)\subseteq \mathrm{cod}(G)$,
    and so Lemma \ref{lem: frob, abel ker, general case, cod} implies that $|\mathrm{cod}(G)|=4$.	
	If part (2) holds,
    a direct computation via $\mathsf{GAP}$ \cite{gap} shows $|\mathrm{cod}(G)|=4$.
\end{proof}

\begin{cor}\label{cor: height=2, |cod(G/N)|=3}
		Assume Hypothesis \ref{hy: height=2, |cod(G/N)|=3} and that $|\mathrm{cod}(G)|=4$.
	If $p^{2}\mid n$ for some $n\in \mathrm{cod}(G)$, then $n=n_p$.
\end{cor}
\begin{proof}
		 Let $n\in \mathrm{cod}(G)$ satisfy $p^{2}\mid n$.
	By Theorem \ref{thm: height=2, |cod(G/N)|=3, classification},
	either $G$ is one of the Frobenius groups listed in part (1) of Theorem \ref{thm: height=2, |cod(G/N)|=3, classification}, or $G$ is one of the groups listed in part (2) of Theorem \ref{thm: height=2, |cod(G/N)|=3, classification}.
	If the former holds, then $n\in \mathrm{cod}(G/N)=\mathrm{cod}(P)$.
	Consequently, $n=n_p$.
    If the latter holds, then a routine check by $\mathsf{GAP}$ \cite{gap} yields that part (2) holds.
\end{proof}

\begin{hy}\label{hy: height=2, |cod(G/N)|=2}
	Let $G=N \rtimes P$ be a solvable group with Fitting height $2$ where $N$ is the nilpotent residual of $G$ and $P\in \mathrm{Syl}_{p}(G)$.
	Set $C=\mathbf{C}_{P}(N)$.
	Assume that $|\mathrm{cod}(G/N)|=2$ and that $N>1$.	
 \end{hy}
 
Assume Hypothesis \ref{hy: height=2, |cod(G/N)|=2} and that $|\mathrm{cod}(G)|\leq 4$.
Then, by Lemma \ref{lem: |codG|=2}, $P$ is an elementary abelian $p$-group such that $P=P_0\times C$.
Also, part (2) of Lemma \ref{lem: |cod(G)|<=4} implies that $N$ is a $q$-group contained in $\mathbf{F}(G)$.

Let $G=V \rtimes H$ be a finite group where $V$ is a completely reducible $H$-module (possibly of mixed characteristic).
Recall that 
$\mathcal{S}_H(V)$ denotes the set of representatives of the isomorphism classes of irreducible $H$-submodules in $V$.
Therefore, 
$$V=\bigtimes_{W\in \mathcal{S}_H(V)} W(V),$$
where $W(V)$ denotes the $W$-homogeneous part of $V$.

\begin{lem}\label{lem: P/C acts fpf on N/N'}
	Assume Hypothesis \ref{hy: height=2, |cod(G/N)|=2}.
	If $|\mathrm{cod}(G)|\leq 4$,
	then $P/C$ is a cyclic group of order $p$ acting Frobeniusly on $N/N'$.
\end{lem}
\begin{proof}
	Let $G$ be a counterexample of minimal possible order.
	Note that $N$ is a $q$-group with $q\neq p$,
	and hence
	$C=\mathbf{C}_{P}(N)=\mathbf{C}_{P}(N/N')=\mathbf{C}_{P}(N/\Phi(N))$.
	As $P/C$ acts coprimely on the $q$-group $N$, by Lemma \ref{lem: Frobeniusly action},  
    $P/C$ acts Frobeniusly on $N/N'$ if and only if $P/C$ acts Frobeniusly on $N/\Phi(N)$.
	By the minimality of $G$, we deduce that $C=N'=\Phi(N)=1$.
	In particular, $N=\mathbf{F}(G)$ is an elementary abelian $q$-group.
	Since $h(G)\geq 2$, Lemma \ref{lem: |codG|=2} forces $|\mathrm{cod}(G)|\geq 3$.
	If $|\mathrm{cod}(G)|= 3$, then we conclude a contradiction by \cite[Theorem 3.4]{alizadeh2019}
	and \cite[Theorem 0.1]{alizadeh2022}.
	Thus, $|\mathrm{cod}(G)|=4$.
    If $N$ is a homogeneous $P$-module over $\mathbb{F}_q$,
	then $\mathbf{C}_{P}(V)=\mathbf{C}_{P}(N)=1$ for each minimal normal subgroup $V$ of $G$ in $N$,
	so $P\cong \mathsf{C}_{p}$ acts Frobeniusly on $N$, a contradiction. 
    Hence, $N$ is not a homogeneous $P$-module over $\mathbb{F}_q$.
	Applying part (2) of Lemma \ref{lem: large orbit} to $G$, we deduce that 
    $N=V(N)\times W(N)$ where $V$ and $W$ are non-isomorphic $P$-submodules of $N$
	such that $|V|=|W|$, and $\{ q^{d},q^{2d}\}\subseteq \mathrm{cod}(G)$ where $q^{d}:=|V|$.
	Since $G'=N$ and $\mathrm{cod}(G/N)=\{1,p\}$, we have $\mathrm{cod}(G|G')=\{q^{d},q^{2d}\}$, contradicting Lemma~\ref{lem: qww}.
\end{proof}

\begin{prop}\label{prop: height=2, |cod(G/N)|=2, C>1}
   Assume Hypothesis \ref{hy: height=2, |cod(G/N)|=2} and that $C>1$. 
   Then $|\mathrm{cod}(G)|=4$ if and only if 
    $G=H\times C$ 
	where $H$ is a Frobenius group with complement $P_0\cong \mathsf{C}_{p}$ and kernel $N$
	such that $N$ is a homogeneous $P_0$-module over $\mathbb{F}_q$,
	and $C$ is an elementary abelian $p$-group.
	Also, if $|\mathrm{cod}(G)|=4$, then $\mathrm{cod}(G)=\{ 1,p,q^{d},pq^{d} \}$ where 
	$d$ is the multiplicative order of $q$ modulo $p$.
\end{prop}
\begin{proof}
    We assume first that $|\mathrm{cod}(G)|=4$.
	As $G=N \rtimes P$ where $P=P_0\times C$ is elementary abelian,
	we have $G=H\times C$ where $H:=NP_0$ and $C$ is an elementary abelian $p$-group.
    Note that $H/N'$ is a Frobenius group with complement $P_0N'/N'\cong \mathsf{C}_{p}$ by Lemma \ref{lem: P/C acts fpf on N/N'},
	and hence $P_0\cong P_0N'/N'\cong \mathsf{C}_{p}$.
	Since $N$ is the nilpotent residual of $G$,
	it is also the nilpotent residual of $H$.
	Therefore, $N=H'$.

	Let $\alpha\in\mathrm{Irr}(H)$ be nonlinear.
	If $\mathrm{cod}(\alpha)=p$,
	as $|H/\ker(\alpha)|=p\cdot |N\ker(\alpha)/\ker(\alpha)|$,
	we have that $\alpha(1)=|N\ker(\alpha)/\ker(\alpha)|$ 
	is the order of a normal Sylow $q$-subgroup of $H/\ker(\alpha)$,
	hence $\alpha(1)=1$, a contradiction. 
	Thus, $q\mid \mathrm{cod}(\alpha)$.
	Now, we claim that $\mathbf{Z}(\alpha)/\ker(\alpha)$ is a $q$-group.
	In fact, otherwise $P_0\ker(\alpha)/\ker(\alpha)$ is a nontrivial Sylow $p$-subgroup of the cyclic group $\mathbf{Z}(\alpha)/\ker(\alpha)$,
	and so $\mathsf{C}_{p}\cong P_0\ker(\alpha)/\ker(\alpha)\unlhd H/\ker(\alpha)$;
    setting $\overline{H}=H/\ker(\alpha)$,
    we have $\overline{H}=\overline{N}\times \overline{P_0}$ where $\mathsf{C}_{p}\cong\overline{P_0}\in \mathrm{Syl}_{p}(\overline{H})$;
	as $\alpha(1)>1$,
	$\overline{N}$ is nonabelian; 
    since $|\mathrm{cod}(\overline{P_0})|=2$ and, 
	by Lemma \ref{lem: |codG|=2}, $|\mathrm{cod}(\overline{N})|>2$,
	we have $|\mathrm{cod}(G)|\geq |\mathrm{cod}(\overline{H})|=|\mathrm{cod}(\overline{N})|\cdot |\mathrm{cod}(\overline{P_0})|>4$ by part (3) of Lemma \ref{lem: direct product, kernel and codegrees}, a contradiction.
	Let $\gamma\in \mathrm{Irr}(C)^\sharp$.
	Then $\mathbf{Z}(\gamma)=C$ is an elementary abelian $p$-group and $\mathrm{cod}(\gamma)=p$.
	As $(|\mathbf{Z}(\alpha)/\ker(\alpha)|,|\mathbf{Z}(\gamma)|)=1$,
	 part (2) of Lemma \ref{lem: direct product, kernel and codegrees} yields $p \cdot \mathrm{cod}(\alpha)=\mathrm{cod}(\alpha)\mathrm{cod}(\gamma)\in \mathrm{cod}(G)$.
	Since $q\mid \mathrm{cod}(\alpha)$,
	we have $\mathrm{cod}(G)=\{ 1,p,\mathrm{cod}(\alpha),p\cdot \mathrm{cod}(\alpha) \}$.
	In particular, $\mathrm{cod}(\alpha)$ is a constant for each nonlinear character $\alpha \in \mathrm{Irr}(H)$.
	As $H$ is not nilpotent, we conclude by \cite[Theorem A]{qian2023} that 
    $H$ is a Frobenius group with complement $P_0\cong \mathsf{C}_{p}$ and elementary abelian kernel $N$
	such that $N$ is a homogeneous $P_0$-module over $\mathbb{F}_q$.
	So, $\mathrm{cod}(G)=\{ 1,p,q^{d}, pq^{d} \}$ where $q^d$ is the order of a $G$-chief factor in $N$ and,
	by part (1) of Lemma \ref{lem: size of chief factor in Frob ker}, 
    $d$ is the multiplicative order of $q$ modulo $p$.
	
	Conversely, we assume that $G=H\times C$ where $H$ is a Frobenius group with complement $P_0\cong \mathsf{C}_{p}$ and kernel $N$
	such that $N=H'$ is a homogeneous $P_0$-module over $\mathbb{F}_q$,
	and $C$ is an elementary abelian $p$-group.
	By \cite[Theorem A]{qian2023}, $\mathrm{cod}(H|N)=\{ q^{d} \}$ with $q^{d}$ the order of an $H$-chief factor in $N$.
    Let $\chi\in \mathrm{Irr}(G)^\sharp$.
    If $\chi(1)=1$, as $G/\ker(\chi)$ is an elementary abelian $p$-group,
	we conclude that $\mathrm{cod}(\chi)=p$.
    Assume that $\chi(1)>1$.
	Then $\chi=\alpha\times \beta$ where $\alpha\in \mathrm{Irr}(H|N)$ and $\beta\in \mathrm{Irr}(C)$.
    Note that $\mathbf{Z}(\alpha)$ is a $q$-group and that $\mathbf{Z}(\beta)$ is a $p$-group,
	and hence part (2) of Lemma \ref{lem: direct product, kernel and codegrees} implies that $\mathrm{cod}(\alpha)\mathrm{cod}(\beta)\in \mathrm{cod}(G)$.
    Therefore, $\mathrm{cod}(G)=\{ 1,p,q^{d},pq^{d} \}$.
\end{proof}

\begin{lem}\label{lem: height=2, |cod(G/N)|=2, C=1, Frob}
	Assume Hypothesis \ref{hy: height=2, |cod(G/N)|=2} and that $C=1$.
	If $|\mathrm{cod}(G)|=4$, then 
	$P\cong \mathsf{C}_{p}$ acts Frobeniusly on the abelian $q$-group
    $N/N'$ 
	such that one of the following holds.
	\begin{description}
		\item[(1)] $N/N'$ is a homogeneous $P$-module over $\mathbb{F}_q$, and $\mathrm{cod}(G/N')=\{ 1,p,q^{d} \}$ where $q^{d}$ is the order of a $G$-chief factor in $N/N'$.
		Moreover, either $\mathbf{C}_{N}(P)>1$, or $G$ is a Frobenius group with complement $P\cong \mathsf{C}_{p}$.
		\item[(2)] Either $\exp(N/N')= q^{2}$ and all $G$-chief factors in $N/N'$ are isomorphic as $P$-modules, or $N/N'$ is an elementary abelian $q$-group such that $\mathcal{S}_{P}(N/N')=\{ U,W \}$ with $|U|=|W|$. 
		In both cases, $G$ is a Frobenius group with complement $P\cong \mathsf{C}_{p}$ 
		and kernel $N$ of nilpotency class at most $2$, and 
		$\mathrm{cod}(G)=\mathrm{cod}(G/N')=\{ 1,p,q^{d},q^{2d} \}$ where $q^{d}$ is the order of a $G$-chief factor in $N/N'$.
	\end{description}
\end{lem}
\begin{proof}
	By Lemma \ref{lem: P/C acts fpf on N/N'},
	$P\cong \mathsf{C}_{p}$ acts Frobeniusly on the abelian $q$-group
    $N/N'$.
	Set $\overline{G}=G/N'$.

	 Assume that $\overline{N}$ is a homogeneous $P$-module over $\mathbb{F}_q$.
	 Let $q^{d}$ be the order of a $G$-chief factor in $\overline{N}$.
	 As $1,p\in \mathrm{cod}(\overline{G})$, 
	 Lemma \ref{lem: frob, abel ker, general case, cod}
	 yields $\mathrm{cod}(\overline{G})=\{ 1,p,q^{d} \}$.
	 If $\mathbf{C}_{N}(P)=1$, as $P\cong \mathsf{C}_{p}$, it follows that
	 $G$ is a Frobenius group with complement $P\cong \mathsf{C}_{p}$.

	 Assume that $\exp(\overline{N})=q$ but $\overline{N}$ is not a homogeneous $P$-module over $\mathbb{F}_q$.
	 So, part (2) of Lemma \ref{lem: large orbit} forces
	 $\mathcal{S}_{P}(\overline{N})=\{ U,W \}$ with $|U|=|W|$,
	 and we obtain 
	 $\mathrm{cod}(G)=\mathrm{cod}(\overline{G})=\{ 1,p,q^{d},q^{2d} \}$ where $q^{d}:=|U|$.

	 Assume that $\exp(\overline{N})>q$.
     Then $\overline{N}=\overline{E}\times \overline{D}$,
	 where $\overline{D}\unlhd\overline{G}$
	 and $\overline{E}$ is a $\overline{P}$-indecomposable abelian subgroup of $\overline{N}$
	 with $\exp(\overline{E})=\exp(\overline{N})$.
	 Consequently, $\overline{EP}$ is isomorphic to $\overline{G}/\overline{D}$. 
	 Since $1,p\in \mathrm{cod}(\overline{G}/\overline{D})$ and $|\mathrm{cod}(\overline{G}/\overline{D})|\leq 4$,
	 Lemma \ref{lem: frob, abel ker, general case, cod} forces that
     all $\overline{G}$-chief factors in $\overline{E}$
	 are isomorphic as $P$-modules.
	 Moreover, $\exp(\overline{N})=\exp(\overline{E})=q^{2}$,
	 and $\mathrm{cod}(G)=\mathrm{cod}(\overline{G})=\mathrm{cod}(\overline{G}/\overline{D})=\{ 1,p,q^{d},q^{2d}\}$ where $q^{2d}=|\overline{E}|$ and $q^{d}$ is the order of a $\overline{G}$-chief factor in $\overline{E}$.
	 We claim next that every $\overline{G}$-chief factor in $\overline{N}$ is isomorphic as a $P$-module.
	  Let $\overline{G}$ be a counterexample of minimal possible order.
	  As $|\mathrm{cod}(\overline{G}/\overline{D})|=4$, the minimality of $\overline{G}$ forces $\overline{D}$ 
	  to be a minimal normal subgroup of $\overline{G}$ that is not isomorphic to $\Omega_1(\overline{E})$ as a $P$-module.
	  In particular, $\overline{D}$ and $\Omega_1(\overline{E})$ are the only minimal normal subgroups of $\overline{G}$.
	  Let $\epsilon\in \mathrm{Irr}(\overline{E})$ be of order $q^{2}$, $\delta \in \mathrm{Irr}(\overline{D})^\sharp$ and $\chi\in \mathrm{Irr}(\overline{G}|\epsilon\times \delta)$.
	  Since $\ker(\chi)\cap \overline{E}=\ker(\chi)\cap \overline{D}=1$,
	  we have $\ker(\chi)=1$.
	  Note also that $\chi(1)=p$, and therefore by calculation $\mathrm{cod}(\chi)=|\overline{E}|\cdot |\overline{D}|=q^{3d}\notin \mathrm{cod}(\overline{G})$, a contradiction.

	 Suppose that $\mathrm{cod}(G)=\mathrm{cod}(G/N')=\{1,p, q^{d}, q^{2d} \}$ with $q^{d}$ the order of a $G$-chief factor in $N/N'$.
     Then $\mathrm{cod}(G|G')\subseteq \{ q^{d},q^{2d} \}$.
	 Indeed, otherwise, $p=\mathrm{cod}(\chi)$ for some $\chi \in \mathrm{Irr}(G|G')$;
	 since $|G|_p=p$, it follows that $\ker(\chi)\leq N$ and $\chi(1)=|N:\ker(\chi)|$;
	 let $\theta$ be an irreducible constituent of $\chi_N$, and observe that $p\nmid \chi(1)$;
	 so, $\chi_N=\theta$ and therefore $\theta(1)=|N:\ker(\theta)|$ which contradicts $\chi(1)=\theta(1)>1$.
      Hence, Lemma \ref{lem: qww} forces $G$ to be a Frobenius group with complement $P\cong \mathsf{C}_{p}$ and kernel $N$.
	 We next show that $c(N)\leq 2$.
     To see that, let $G$ be a counterexample of minimal possible order.
	 As $N=\mathbf{F}(G)$,
	 every minimal normal subgroup of the solvable $G$ is contained in $N$.
     We claim that $|\mathrm{cod}(G/D)|=4$ for each minimal normal subgroup $D$ of $G$.
	 Otherwise, there is some minimal normal subgroup $D$ of $G$ such that $|\mathrm{cod}(G/D)|\leq 3$; 
	 as $c(N)>2$, $N/D$ is a nontrivial $q$-group;
	 note that $PD/D$ is also a nontrivial $p$-group, and so part (3) of Lemma \ref{lem: direct product, kernel and codegrees} implies that $h(G/D)\neq 1$;
	 as $h(G/D)\leq h(G)=2$, we must have $h(G/D)=2$, whence $|\mathrm{cod}(G/D)|=3$; 
	 now \cite[Theorem 3.4]{alizadeh2019} and \cite[Theorem 0.1]{alizadeh2022}
	 forces $N/D$ to be abelian, a contradiction.
	 Consequently, the minimality of $G$ yields that $G$ has a unique minimal normal subgroup, say $D$, and $c(N/D)=2$.
	 Let $\theta\in \mathrm{Irr}(N|D)$ and $\chi\in \mathrm{Irr}(G|\theta)$.
	 Then $\chi(1)=p\cdot \theta(1)$ and $\ker(\chi)=1$,
	 so $\mathrm{cod}(\chi)=|N|/\theta(1)$.
	 As $G$ is a Frobenius group with complement $P\cong \mathsf{C}_{p}$ and kernel $N$ a $q$-group, every $G$-chief factor in $N$ shares the same order $q^{d}$ by part (1) of Lemma \ref{lem: size of chief factor in Frob ker}.
      Hence, $|N:D|=q^{sd}$ for some positive integer $s$ and $|D|=q^{d}$.
	 Recall that $\mathrm{cod}(G)=\mathrm{cod}(G/N')=\{ 1,p,q^{d},q^{2d} \}$,
	 and so $|N|/\theta(1)=\mathrm{cod}(\chi)\leq q^{2d}$, implying $\theta(1)\geq q^{sd-d}$.
	 As $D\leq \mathbf{Z}(N)$, we have $\theta(1)^{2}\mid |N:D|$, so $2(sd-d)\leq sd$.
	 By calculation, $s\leq 2$.
	 Since $c(N/D)=2$, we must have $s=2$ and $D<N'$.
     Then $|N:D|=q^{2d}$ which contradicts $\mathrm{cod}(G/N')=\{ 1,p,q^{d},q^{2d} \}$.
\end{proof}

Let a finite group $P$ act on a finite group $N$.
We set $\mathrm{Irr}_P(N)=\{ \theta\in \mathrm{Irr}(N):\theta^x=\theta,~\text{for all}~x\in P \}$.

\begin{lem}\label{lem: height=2, |cod(G/N)|=2, C=1, not Frob, cN=2}
	Assume Hypothesis \ref{hy: height=2, |cod(G/N)|=2} and that $C=1$.
    Assume also that $\mathbf{C}_{N}(P)>1$ and that $c(N)=2$.
    If $|\mathrm{cod}(G)|=4$,
	then $P\cong \mathsf{C}_{p}$, and $N=\mathbf{F}(G)$ is a semi-extraspecial $q$-group such that $N'=\mathbf{Z}(G)$ and $N/N'$ is a homogeneous $P$-module over $\mathbb{F}_q$.
	Moreover, $\mathrm{cod}(G)=\{ 1,p,q^{d}, pq \sqrt{|N:N'|} \}$ where 
	$q^{d}$ is the order of a $G$-chief factor in $N/N'$.   
\end{lem}
\begin{proof}
	By Lemma \ref{lem: height=2, |cod(G/N)|=2, C=1, Frob},
	$P\cong \mathsf{C}_{p}$ acts Frobeniusly on $N/N'$ such that $N/N'$ is a homogeneous $P$-module over $\mathbb{F}_q$, and $\mathrm{cod}(G/N')=\{ 1,p,q^{d} \}$ where $q^{d}$ is the 
	order of a $G$-chief factor in $N/N'$.
	In particular, $\mathbf{C}_{N}(P)\leq N'=\Phi(N)$.
    Also, an application of part (1) of Lemma \ref{lem: P act fpf on N/N'} yields that $N=\mathbf{F}(G)\in \mathrm{Syl}_{q}(G)$ is the unique
	maximal normal subgroup of $G$.
    As $\mathbf{C}_{N}(P)>1$, Glauberman's correspondence \cite[Theorem 13.1]{isaacs1994} implies that $|\mathrm{Irr}_P(N)|=|\mathrm{Irr}(\mathbf{C}_{N}(P))|>1$.
	The Frobenius action of $P$ on $N/N'$ forces $\mathrm{Irr}_P(N)^\sharp \subseteq \mathrm{Irr}(N|N')$.
    Let $\varphi\in \mathrm{Irr}_P(N)^\sharp$, $\omega\in \mathrm{Irr}(G|\varphi)$, and observe that $|G:N|=p$ is a prime.
    It follows that $\omega_N=\varphi$.
	As $N$ is the unique maximal normal subgroup of $G$,
	$\ker(\omega)=\ker(\varphi)<N$.
    Note that 
	$$\mathrm{cod}(\omega)=p\cdot \frac{|N:\ker(\varphi)|}{\varphi(1)}=p\cdot \mathrm{cod}(\varphi)\notin \mathrm{Irr}(G/N'),$$
	and so $\mathrm{cod}(\varphi)$ is a constant for each $\varphi \in \mathrm{Irr}_P(N)^\sharp$.
    Write $\mathrm{cod}(\varphi)=q^{l}$ where $l>0$.
    Then $\mathrm{cod}(G)=\{ 1,p,q^{d},pq^{l} \}$.

    Since $c(N)=2$,
	we have $N'=\Phi(N)\leq \mathbf{Z}(N)$,
	hence $\exp(N')=\exp (N/\mathbf{Z}(N))=q$ and $N'$ is an elementary abelian $q$-group. 
	Also, $N'=\mathbf{C}_{N'}(P)\times [N',P]$ and $\mathbf{C}_{N'}(P)\leq\mathbf{Z}(G)$.
	As $P\cong \mathsf{C}_{p}$ acts Frobeniusly on both $N/N'$ and $[N',P]$,
	we have that $\mathbf{C}_{N'}(P)=\mathbf{Z}(G)$,
	and $G/\mathbf{C}_{N'}(P)$ is a Frobenius group with kernel $N/\mathbf{C}_{N'}(P)$.

We claim that $\mathrm{Irr}_P(N)^\sharp=\mathrm{Irr}(N|N')$.
	Assume not. 
	Let $\theta\in \mathrm{Irr}(N|N')-\mathrm{Irr}_P(N)$ and $\chi \in \mathrm{Irr}(G|\theta)$.
	As $|G:N|=p$ is a prime, $\chi=\theta^{G}$.
	Note that $\mathrm{cod}(\chi)=q^{k}$ for some $k>0$,
	and therefore $\mathrm{cod}(\chi)=q^{d}$.
	Since $|G/\ker(\chi)|<\mathrm{cod}(\chi)^{2}=q^{2d}$,
	$N/N'\ker(\chi)$ is a $G$-chief factor of order $q^{d}$
	and $|N'\ker(\chi)/\ker(\chi)|<q^{d}$ where $d>1$.
	Set $\overline{G}=G/\ker(\chi)$.
	By part (2) of Lemma \ref{lem: P act fpf on N/N'}, we have $\overline{N}'=\mathbf{C}_{\overline{N}}(\overline{P})=\mathbf{Z}(\overline{G})$.
	Take $\psi\in \mathrm{Irr}_{\overline{P}}(\overline{N})^\sharp$.
	Then $\ker(\psi)\unlhd\overline{G}$.
	Since $\overline{N}/\ker(\psi)$ is nonabelian
	and $\overline{N}/\overline{N}'\ker(\psi)$ is a nontrivial irreducible $\overline{P}$-module over $\mathbb{F}_q$,
	we have $\ker(\psi)\leq \overline{N}'$.
    Because $\overline{N}'$ is an elementary abelian $q$-subgroup of $\mathbf{Z}(N)$, 
	$|\overline{N}':\ker(\psi)|=|\overline{N}':\ker(\psi_{\overline{N}'})|=q$.
    As $\overline{N}/\overline{N}'$ is a $\overline{G}$-chief factor,
	$\overline{N}/\ker(\psi)$ is an extraspecial $q$-group of order $q^{d+1}$.
	In particular, $\psi(1)=q^{d/2}$ for each $\psi \in \mathrm{Irr}_{\overline{P}}(\overline{N})^\sharp$.
	However, as 
   $$\sum_{\psi\in \mathrm{Irr}_{\overline{P}}(\overline{N})^\sharp}\psi(1)^{2}=q^{d}\cdot (|\overline{N}'|-1)=|\overline{N}/\overline{N}'|\cdot (|\overline{N}'|-1)=|\overline{N}|-|\overline{N}/\overline{N}'|=\sum_{\alpha \in \mathrm{Irr}(\overline{N}|\overline{N}')}\alpha(1)^{2}$$
   where the first equality holds as $|\mathrm{Irr}_{\overline{P}}(\overline{N})|=|\mathrm{Irr}(\mathbf{C}_{\overline{N}}(\overline{P}))|=|\overline{N}'|$,
   we conclude that $\mathrm{Irr}_{\overline{P}}(\overline{N})^\sharp=\mathrm{Irr}(\overline{N}|\overline{N}')$ which contradicts the existence of $\theta$.

 Recalling that $G/\mathbf{C}_{N'}(P)$ is a Frobenius group with kernel $N/\mathbf{C}_{N'}(P)$
   and that $\mathrm{Irr}(N|N')=\mathrm{Irr}_P(N)^\sharp$,
   we have $N'=\mathbf{C}_{N'}(P)=\mathbf{Z}(G)$.
   In particular, $|\mathrm{Irr}(N|N')|=|\mathrm{Irr}_P(N)^\sharp|=|\mathrm{Irr}(N')^\sharp|=|N'|-1$.
   Recall that $\varphi \in \mathrm{Irr}_P(N)^\sharp=\mathrm{Irr}(N|N')$ and $\omega\in \mathrm{Irr}(G|\varphi)$ such that $\omega_N=\varphi$ and $\ker(\omega)=\ker(\varphi)$.
   Applying Lemma \ref{lem: semi-extra},
   we conclude that $N$ is a semi-extraspecial $q$-group and $\varphi(1)=\sqrt{|N:N'|}$.
   So, $\varphi=\frac{1}{\varphi(1)}\lambda^N$ for some $\lambda\in \mathrm{Irr}(N')^\sharp$ (see e.g.  \cite[Problem 6.3]{isaacs1994}). 
   Therefore, $\ker(\omega)=\ker(\varphi)=\ker(\varphi_{N'})=\ker(\lambda)$ is a maximal subgroup 
   of the elementary abelian $q$-group $N'$.
   Consequently, $\mathrm{cod}(\omega)=pq \sqrt{|N:N'|}$,
   and therefore, $\mathrm{cod}(G)=\{ 1,p,q^{d}, pq \sqrt{|N:N'|} \}$.
\end{proof}

\begin{prop}\label{prop: height=2, |cod(G/N)|=2, C=1, not Frob}
	Assume Hypothesis \ref{hy: height=2, |cod(G/N)|=2} and that $C=1$.
    Assume that 
	$\mathbf{C}_{N}(P)>1$.
    Then $|\mathrm{cod}(G)|=4$ if and only if $P\cong \mathsf{C}_{p}$, and $N$ is a semi-extraspecial $q$-group such that $N'=\mathbf{Z}(G)$ and $N/N'$ is a homogeneous $P$-module over $\mathbb{F}_q$.
	Also, if $|\mathrm{cod}(G)|=4$, then $\mathrm{cod}(G)=\{ 1,p,q^{d}, pq \sqrt{|N:N'|} \}$ where 
	$d$ is the multiplicative order of $q$ modulo $p$.  
\end{prop}
\begin{proof}
	We assume first that $|\mathrm{cod}(G)|=4$.
    By Lemma \ref{lem: height=2, |cod(G/N)|=2, C=1, Frob},
    $G/N'$ is a Frobenius group with complement $PN'/N\cong \mathsf{C}_{p}$ and kernel $N/N'$, and $\mathrm{cod}(G/N')=\{ 1,p,q^{d} \}$ where $q^{d}$ is the order of a $G$-chief factor in $N/N'$.
    Hence, $G'=N$ and $\mathrm{cod}(G/N'|G'/N')=\{ q^d \}$.
    By part (1) of Lemma \ref{lem: size of chief factor in Frob ker}, $d$ is the multiplicative order of $q$ modulo $p$.
    Since $\mathbf{C}_{N}(P)>1$, Lemma \ref{lem: qww} forces $q^{d}$ to be the unique $q$-power
	in $\mathrm{cod}(G|G')$.
	Moreover, $N=\mathbf{F}(G)$ is the unique maximal normal subgroup of $G$ by Lemma \ref{lem: P act fpf on N/N'}.
 By Lemma \ref{lem: height=2, |cod(G/N)|=2, C=1, not Frob, cN=2}, it remains to show that $c(N)=2$.
 Let $G$ be a counterexample of minimal possible order.
  Then $c(N)\geq 3$.
  Let $D$ be a minimal normal subgroup of $G$.
  Then $D\leq \mathbf{F}(G)=N$.
    As $c(N)\geq 3$, $N/D$ is nonabelian.
	We now proceed in the next three steps to conclude a contradiction.

 \textbf{Step 1.} $h(G/D)=2$ and $|\mathrm{cod}(G/D)|=4$.

  We first show that $h(G/D)\geq 2$.
  In fact, otherwise $G/D=N/D\times PD/D$ where $N/D\in \mathrm{Syl}_{q}(G)$ is nonabelian and $PD/D\cong \mathsf{C}_{p}$;
  so, $|\mathrm{cod}(G/D)|\geq 6$ by Lemma \ref{lem: |codG|=2} and part (3) of Lemma \ref{lem: direct product, kernel and codegrees}, a contradiction.
  Since $h(G/D)\leq h(G)=2$, we have $h(G/D)=2$.
  We next prove that $|\mathrm{cod}(G/D)|=4$.
  Indeed, otherwise $|\mathrm{cod}(G/D)|\leq 3$;
  as $h(G/D)=2$, \cite[Theorem 3.4]{alizadeh2019}
  and \cite[Theorem 0.1]{alizadeh2022} forces $N/D$ to be abelian, a contradiction.

 \textbf{Step 2.} $\mathbf{C}_{N/D}(PD/D)>1$.

  Otherwise, as $PD/D\cong \mathsf{C}_{p}$, $G/D$ is a Frobenius group with kernel $N/D$ and complement $PD/D$.
  For $\chi\in \mathrm{Irr}(G/D|N'D/D)$,
   we have $D\leq \ker(\chi)< N$,
  so $\overline{G}:=G/\ker(\chi)$ is a Frobenius group with nonabelian kernel $\overline{N}=N\ker(\chi)/\ker(\chi)$.
  Note that $\mathrm{cod}(\chi)=q^{k}$ for some $k>0$,
  and hence $\mathrm{cod}(\chi)=q^{d}$. 
 Since $|\overline{G}|<\mathrm{cod}(\chi)^{2}=q^{2d}$,
	we deduce that $\overline{N}/\overline{N}'$ is a $G$-chief factor of order $q^{d}$
	and $|\overline{N}'|<q^{d}$ with $d>1$.
	So, part (2) of Lemma \ref{lem: P act fpf on N/N'} implies that $\mathbf{C}_{\overline{N}}(\overline{P})=\overline{N}'>1$,
	which contradicts the fact that $\overline{G}$ is Frobenius group with kernel $\overline{N}$.
 
	\textbf{Step 3.} Conclude a contradiction.

    By Steps 1 and 2, $G/D$ satisfies the hypotheses of the proposition.
    The minimality of $G$ then implies that $D$ is the unique minimal normal subgroup of $G$ and $c(N/D)=2$.
	Hence, $D\leq N'\cap\mathbf{Z}(N)$ and so $c(N)=3$. 
	Applying Lemma \ref{lem: height=2, |cod(G/N)|=2, C=1, not Frob, cN=2} to $G/D$,
	we deduce that 
	$\mathrm{cod}(G)=\mathrm{cod}(G/D)=\{ 1,p,q^{d},pq\sqrt{|N:N'|} \}$, $N/D$ is a semi-extraspecial $q$-group with $N'/D=\mathbf{Z}(G/D)$, and $\mathrm{Irr}_{P}(N/D)^\sharp=\mathrm{Irr}(N/D|N'/D)$ has size $|N'/D|-1$.

  Let $\chi \in \mathrm{Irr}(G|D)$ and $\theta$ an irreducible constituent of $\chi_N$.
   Since $D$ is the unique minimal normal subgroup of $G$,
   we have $\ker(\chi)=1$.
   Because $q\mid \mathrm{cod}(\theta)\mid \mathrm{cod}(\chi)$,
   $\mathrm{cod}(\chi)$ equals either $q^{d}$ or $pq \sqrt{|N:N'|}$.
  Assume that $\mathrm{cod}(\chi)=pq \sqrt{|N:N'|}$.
  Then $p\nmid \chi(1)$, and hence $\chi_N=\theta$.
  As $\mathrm{cod}(\chi)=\frac{p|N|}{\theta(1)}=pq \sqrt{|N:N'|}$,
   we have $\theta(1)^{2}=\frac{|N|\cdot |N'|}{q^{2}}$.
 However, $\theta(1)^{2}\leq |N:\mathbf{Z}(N)|\leq |N:D|\leq |N|/q$ which forces $|N'|\leq q$,
 contradicting $c(N)=3$.
  So, $\mathrm{cod}(G|D)=\{ q^{d} \}$.  
  Moreover, $\mathrm{Irr}(N|D)\cap \mathrm{Irr}_P(N)=\varnothing$.
  Recall that $\mathrm{Irr}_P(N/D)^\sharp=\mathrm{Irr}(N/D|N'/D)$ has size $|N'/D|-1$,
  and hence  $\mathrm{Irr}_P(N)^\sharp=\mathrm{Irr}(N/D|N'/D)$.
  By calculation,
  \begin{equation}\label{eq_2}
	|\mathrm{Irr}(N'/D)|-1=|N'/D|-1= |\mathrm{Irr}(N/D|N'/D)|=|\mathrm{Irr}_P(N)|-1=|\mathrm{Irr}(\mathbf{C}_{N}(P))|-1
\end{equation}
  where the first equality holds because $N'/D$ is abelian, while the third holds by Glauberman's correspondence.
  Since $|G|=|G:\ker(\chi)|<\mathrm{cod}(\chi)^{2}=q^{2d}$,
  we deduce that $N/N'$ is a $G$-chief factor of order $q^{d}$ and $|N'|<q^{d}$.
  So, part (2) of Lemma \ref{lem: P act fpf on N/N'} forces
   $N'=\mathbf{C}_{N}(P)$,
   whence $|\mathrm{Irr}(N'/D)|=|\mathrm{Irr}(N')|$ by (\ref{eq_2}), a contradiction.

   Conversely, we assume that $P\cong \mathsf{C}_{p}$, and $N=\mathbf{F}(G)$ is a semi-extraspecial $q$-group such that $N'=\mathbf{Z}(G)$ and $N/N'$ is a homogeneous $P$-module over $\mathbb{F}_q$.
   Then $G/N'$ is a Frobenius group with kernel $N/N'$.
   Moreover, $N$ is the unique maximal normal subgroup of $G$ by Lemma \ref{lem: P act fpf on N/N'}.
   As $1,p\in \mathrm{cod}(P)\subseteq \mathrm{cod}(G/N')$,
  Lemma \ref{lem: frob, abel ker, general case, cod} gives $\mathrm{cod}(G/N')=\{ 1,p,q^{d} \}$ where $q^{d}$ is the order of a $G$-chief factor in $N/N'$.
   Since $N$ is semi-extraspecial,
   Lemma \ref{lem: semi-extra} yields $\theta(1)=\sqrt{|N:N'|}$ for each $\theta\in \mathrm{Irr}(N|N')$.
   Let $\lambda$ be an irreducible constituent of $\theta_{N'}$, and observe that $\theta=\frac{1}{\theta(1)}\lambda^N$. 
   As $N'=\mathbf{Z}(G)$, $\lambda$ is $G$-invariant
   and so is $\theta$.
   For every $\chi\in \mathrm{Irr}(G|\theta)$,
   we have that $\chi_N=\theta$ and $\ker(\chi)\leq N$, so $\ker(\chi)=\ker(\chi_{N'})=\ker(\lambda)$ is a maximal subgroup of the elementary abelian $q$-group $N'$.
   A routine calculation gives $\mathrm{cod}(\chi)=pq\sqrt{|N:N'|}$.
   Thus, $\mathrm{cod}(G)=\{ 1,p,q^{d},pq\sqrt{|N:N'|} \}$.
\end{proof}

\begin{thm}\label{thm: height=2, |cod(G/N)|=2, classification}
	Assume Hypothesis \ref{hy: height=2, |cod(G/N)|=2} and that $C=1$.
	Then $|\mathrm{cod}(G)|=4$ if and only if one of the following holds.
	\begin{description}
		\item[(1)] $G$ is a Frobenius group with kernel $N \in \mathrm{Syl}_{q}(G)$ and complement $P\cong\mathsf{C}_{p}$, 
		 and one of the following holds.
		\begin{description}
			\item[(1a)] $N$ is an abelian group of exponent $q^{2}$, and all $G$-chief factors in $N$ are isomorphic as a $P$-module.
			\item[(1b)] $N$ is an elementary abelian $q$-group, and there are exactly two non-isomorphic $P$-modules among all $G$-chief factors in $N$. 
			\item[(1c)] $N/\ker(\theta^{G})$ is an ultraspecial $q$-group of order $q^{3d}$ for each nonlinear $\theta\in \mathrm{Irr}(N)$ where $d$ is the multiplicative order of $q$ modulo $p$. 
			Also, either $N/N'$
			is an abelian group of exponent $q^{2}$ and all $G$-chief factors in $N/N'$ are isomorphic as a $P$-module,
            or $N/N'$ is an elementary abelian $q$-group and there are exactly two non-isomorphic $P$-modules among all $G$-chief factors in $N/N'$.
			\item[(1d)] $c(N)\geq 2$ and $N/N'$ is a homogeneous $P$-module over $\mathbb{F}_q$. Also, for each nonlinear $\theta\in \mathrm{Irr}(N)$, 
			there exists a positive integer $k$ such that
			$|N:\ker(\theta^G)|/\theta(1)=q^{k}>q^{d}$ where $d$ is the multiplicative order of $q$ modulo $|P|$.
		\end{description}
		\item[(2)] $P\cong \mathsf{C}_{p}$, and $N$ is a semi-extraspecial $q$-group such that $N'=\mathbf{Z}(G)$ and $N/N'$ is a homogeneous $P$-module over $\mathbb{F}_q$.
	\end{description}
\end{thm}
\begin{proof}
	Assume that $|\mathrm{cod}(G)|=4$.
	By Lemma \ref{lem: height=2, |cod(G/N)|=2, C=1, Frob},
	$P\cong \mathsf{C}_{p}$ acts Frobeniusly on the abelian $q$-group
    $N/N'$.
	If $\mathbf{C}_{N}(P)>1$,
	then Proposition \ref{prop: height=2, |cod(G/N)|=2, C=1, not Frob} yields part (2).
    Suppose that $\mathbf{C}_{N}(P)=1$.
	Then $G$ is a Frobenius group with complement $P\cong \mathsf{C}_{p}$ and kernel $N=G'$ a $q$-group.
	In particular, every element in $\mathrm{cod}(G|G')$ is a power of $q$.
    If $N$ is abelian,
	then Lemma \ref{lem: height=2, |cod(G/N)|=2, C=1, Frob} gives either (1a) or (1b).
	Assume now that $N$ is nonabelian and $N/N'$ is a homogeneous $P$-module over $\mathbb{F}_q$.
	Then, by Lemma \ref{lem: size of chief factor in Frob ker}, every $G$-chief factor in $N$ has order $q^{d}$ where $d$ is the multiplicative order of $q$ modulo $p$.
	As $\mathrm{cod}(G/N')=\{ 1,p,q^{d} \}$ by \cite[Corollary B]{qian2023} and $N'\leq G'$, 
	it follows that $\mathrm{cod}(G|N')=\{ q^{k} \}$ for some positive integer $k$;
	so, $\mathrm{cod}(\theta^{G})=|N:\ker(\theta^G)|/\theta(1)=q^{k}$ for each nonlinear $\theta\in \mathrm{Irr}(N)$.
	Next, we show that $k>d$.
    Assume not. 
	Set $\overline{G}=G/\ker(\theta^{G})$.
	Then $|\overline{G}|<\mathrm{cod}(\theta^{G})^{2}=q^{2k}\leq q^{2d}$.
	Note that $\overline{N}$ is nonabelian.
    As $|\overline{N}|<|\overline{G}|\leq q^{2d}$ and $|\overline{N}:\overline{N}'|\geq q^{d}$,
    we conclude a contradiction that 
	$|\overline{N}'|<q^{d}$.

	So, by Lemmas \ref{lem: height=2, |cod(G/N)|=2, C=1, Frob} and \ref{lem: size of chief factor in Frob ker},
    we may assume that $G$ is a Frobenius group with kernel $N$ of nilpotency class $2$ and complement $P\cong \mathsf{C}_{p}$,
	such that
	$G/N'$ satisfies either (1a) or (1b) and $\mathrm{cod}(G/N')=\{ 1,p,q^{d},q^{2d} \}$ where $q^{d}$ is the order of 
	a $G$-chief factor in $N$ and $d$ is the multiplicative order of $q$ modulo 
	$p$.
	Let $\theta\in \mathrm{Irr}(N|N')$ and $\chi\in \mathrm{Irr}(G|\theta)$.
	Then $\chi=\theta^{G}$, $\mathrm{cod}(\chi)\in \{ q^{d},q^{2d} \}$ and so $\ker(\chi)< N$.
	Set $\overline{G}=G/\ker(\chi)$.
	Then $|\overline{N}|<|\overline{G}|<\mathrm{cod}(\chi)^{2}\leq q^{4d}$.
	Note that $\overline{N}$ is nonabelian as $\theta\in \mathrm{Irr}(\overline{N}|\overline{N}')$,
	and hence $|\overline{N}|\in \{ q^{2d},q^{3d} \}$.

    If $|\overline{N}|=q^{2d}$,
	then $\overline{N}$ is a special $q$-group with $|\mathbf{Z}(\overline{N})|=q^{d}$.
	Since $q^{2}\leq \theta(1)^{2}\leq |\overline{N}: \mathbf{Z}(\overline{N})|=q^{d}$,
    we have $q^{d}<\mathrm{cod}(\chi)<q^{2d}$, a contradiction.
	Hence $|\overline{N}|=q^{3d}$.
	As $|\overline{N}|/\theta(1)=\mathrm{cod}(\chi)\in \{ q^{d}, q^{2d} \}$ 
	and $p|\overline{N}|=|\overline{G}|<\mathrm{cod}(\chi)^{2}$,
	we deduce that $\mathrm{cod}(\chi)=q^{2d}$ and $\theta(1)=q^{d}$.
	Note that $\theta(1)^{2}\leq |\overline{N}:\mathbf{Z}(\overline{N})|$,
	and so $\theta(1)=|\overline{N}:\mathbf{Z}(\overline{N})|^{\frac{1}{2}}$ for all $\theta\in \mathrm{Irr}(\overline{N}|\overline{N}')$.
    Since $c(N)=2$, we have $c(\overline{N})=2$, and hence $\overline{N}'=\mathbf{Z}(\overline{N})$.  
    Therefore, by Lemma~\ref{lem: semi-extra}, $\overline{N}$ is an ultraspecial $q$-group of order $q^{3d}$.

	Conversely, assume that either (1) or (2) holds.
	By our assumptions, we know that
	$G'=N$ and $\mathrm{cod}(G/G')=\{ 1,p \}$.
    So, it suffices to show that $|\mathrm{cod}(G|G')|=2$ and $\mathrm{cod}(G|G')\cap \mathrm{cod}(G/G')=\varnothing$.
    Denote by $d$ the multiplicative order of $q$ modulo $p$.
    As $G/N'$ is a Frobenius group with complement of order $p$ and abelian kernel $N/N'$,
	every $G$-chief factor in $N/N'$ has order $q^{d}$.
    If (1a) holds, then, by Lemma \ref{lem: frob, abel ker, general case, cod}, $\mathrm{cod}(G|G')=\{q^{d}, q^{2d} \}$.
	If (2) holds,
    then we are done by Proposition \ref{prop: height=2, |cod(G/N)|=2, C=1, not Frob}.

	Suppose that (1b) holds.
	Then $G'$ is an elementary abelian $q$-group.
	Let $\mathcal{S}_P(G')=\{ V,W \}$, and write $G'=A\times B$ where $A=V(G')$ and $B=W(G')$ (see the paragraph proceeding Lemma \ref{lem: P/C acts fpf on N/N'}).
	Then $|V|=|W|=q^{d}$.
	So, Lemma \ref{lem: large orbit} forces $\{ q^{d},q^{2d} \}\subseteq \mathrm{cod}(G|G')$.
	Let $\alpha\in \mathrm{Irr}(A)$ and $\beta\in \mathrm{Irr}(B)$ be such that $\alpha\times \beta\neq 1_{G'}$.
    If either $\alpha=1_A$ or $\beta=1_B$, then $\mathrm{cod}((\alpha\times \beta)^{G})=q^{d}$ by Lemma \ref{lem: frob, abel ker, general case, cod}.
	Assume now that $\alpha\neq 1_A$ and $\beta\neq 1_B$.
	Then $\mathrm{cod}(\alpha^G)=|A:\ker(\alpha^G)|=q^{d}=|B:\ker(\beta^{G})|=\mathrm{cod}(\beta^{G})$ by Lemma \ref{lem: frob, abel ker, general case, cod}.
    Note that $\ker(\alpha^G)\times \ker(\beta^{G})\leq \ker((\alpha\times \beta)^{G})$,
	and so $\mathrm{cod}((\alpha\times \beta)^{G})=|AB:\ker((\alpha\times \beta)^{G})|$ divides $|A:\ker(\alpha^G)|\cdot |B:\ker(\beta^{G})|=q^{2d}$.
    Since $\ker((\alpha\times \beta)^{G})\unlhd G$ and every $G$-invariant subgroup of $G'$ has order
	a power of $q^{d}$, $\mathrm{cod}((\alpha\times \beta)^{G})$ is a nontrivial power of $q^{d}$.
	Therefore, $\mathrm{cod}((\alpha\times \beta)^{G})\in \{ q^{d},q^{2d} \}$.
	As a consequence, $\mathrm{cod}(G|G')=\{q^{d}, q^{2d} \}$.

	Suppose that (1c) holds. 
	Then $\mathrm{cod}(G/N'|G'/N')=\{ q^{d},q^{2d} \}$.
	Let $\theta$ be a nonlinear character in $\mathrm{Irr}(N)$.
	As $G$ is a Frobenius group with kernel $N$, $\theta^{G}\in \mathrm{Irr}(G|N')$.
    Note that $N/\ker(\theta^G)$ is an ultraspecial $q$-group of order $q^{3d}$,
    and so $\theta$, as an irreducible character of $N/\ker(\theta^{G})$,
	has degree $q^{d}$ by Lemma \ref{lem: semi-extra}.
	Therefore, $\mathrm{cod}(\theta^{G})=|N/\ker(\theta^{G})|/\theta(1)=q^{2d}$.
	As a consequence, $\mathrm{cod}(G|G')=\{ q^{d},q^{2d} \}$.

	Suppose that (1d) holds.
	Then $\mathrm{cod}(G/N'|G'/N')=\{ q^{d} \}$ by \cite[Theorem A]{qian2023}.
	Let $\theta$ be a nonlinear character in $\mathrm{Irr}(N)$.
	As $G$ is a Frobenius group with kernel $N$, $\theta^{G}\in \mathrm{Irr}(G|N')$.
	Therefore, $\mathrm{cod}(\theta^{G})=|N/\ker(\theta^{G})|/\theta(1)=q^{k}$.
    Consequently, we conclude that $\mathrm{cod}(G|G')=\{ q^{d},q^{k} \}$ where $k>d$.	
\end{proof}

\begin{cor}\label{cor: height=2, |cod(G/N)|=2}
	Assume Hypothesis \ref{hy: height=2, |cod(G/N)|=2}.
	If $|\mathrm{cod}(G)|=4$, then $n_p\leq p$ for each $n\in \mathrm{cod}(G)$.
\end{cor}
\begin{proof}
	If $G$ is a Frobenius group with complement $P\cong \mathsf{C}_{p}$, then we are done.
	If $C>1$, then we are done by Proposition \ref{prop: height=2, |cod(G/N)|=2, C>1}.
	So, 
	we may assume that $G$ satisfies part (2) of Theorem \ref{thm: height=2, |cod(G/N)|=2, classification}.
	Consequently, we are done by Proposition \ref{prop: height=2, |cod(G/N)|=2, C=1, not Frob}.
\end{proof}

\section{Solvable groups with Fitting height 3}

In this section, we classify the finite solvable groups $G$ with Fitting height 3 such that $|\mathrm{cod}(G)|=4$.

\begin{hy}\label{hy: height=3}
	Let $G$ be a solvable group with Fitting height $3$, $K$ the nilpotent residual of $G$ and 
	$V$ the nilpotent residual of $K$.
\end{hy}

\begin{lem}\label{lem: height=3, induction}
	Assume Hypothesis \ref{hy: height=3} and that $|\mathrm{cod}(G)|=4$.
	Then $|\mathrm{cod}(G/D)|=4$ for each proper $G$-invariant subgroup $D$ of $V$.
\end{lem}
\begin{proof}
	Note that $|\mathrm{cod}(G/D)|\leq |\mathrm{cod}(G)|=4$, and hence it suffices to show that $|\mathrm{cod}(G/D)|\geq 4$.
	Otherwise, $|\mathrm{cod}(G/D)|\leq 3$.
	So, we conclude by Lemma \ref{lem: |codG|=2}, \cite[Theorem 3.4]{alizadeh2019} and \cite[Theorem 0.1]{alizadeh2022} that $h(G/D)\leq 2$ which contradicts $h(G/D)=h(G)=3$.
\end{proof}

\begin{lem}\label{lem: height=3 cod=4 V minimal}
	Assume Hypothesis \ref{hy: height=3} and that $V$ is minimal normal in $G$.
	If $|\mathrm{cod}(G)|=4$,
    then $G/V$ is a Frobenius group with cyclic complement of order $p$ and cyclic kernel $K/V$ of order $q$, and $K$ is a Frobenius group with cyclic complement of order $q=\frac{r^{pm}-1}{r^{m}-1}$ and kernel $V\cong (\mathsf{C}_{r})^{pm}$ such that $V$ is minimal normal in $K$. 
	In particular, $\mathrm{cod}(G/V)=\{1,p,q \}$ and $\mathrm{cod}(G)=\{ 1,p,q,pr^{pm} \}$.
\end{lem}
\begin{proof}
	As $|\mathrm{cod}(G/V)|\leq |\mathrm{cod}(G)|\leq  4$ and $h(G/V)=2$,
	it follows by Lemma \ref{lem: |cod(G)|<=4} that $K/V$ is a $q$-group and $G/K$ is a $p$-group of nilpotency class at most 2.
	Let $Q\in \mathrm{Syl}_{q}(K)$.
    Then $K=V \rtimes Q$ because the abelian group $V$ is the nilpotent residual of $K$.	
	By the Frattini's argument, $G=V\mathbf{N}_{G}(Q)$ where $V\cap \mathbf{N}_{G}(Q)=\mathbf{C}_{V}(Q)$.
    The minimality of $V$ then implies $\mathbf{C}_{V}(Q)=1$ (otherwise $K$ would be nilpotent, contradicting $h(K)=2$).
    Hence $G=V \rtimes H$ where $H:=\mathbf{N}_{G}(Q)$ is a maximal subgroup of $G$,
	and so $H=Q \rtimes P$ where $P\in \mathrm{Syl}_{p}(H)$.
    Since $Q$ is the nilpotent residual of $H$ and $\mathbf{C}_{H}(V)<H$,
    $H/\mathbf{C}_{H}(V)$ is not a $q$-group.
	Moreover, Lemma \ref{lem: kernel, abelian, cod} shows that $r^{n}\mid \mathrm{cod}(\chi)$ for each $\chi\in \mathrm{Irr}(G|V)$ where $r^{n}=|V|$.

    We claim next that $|\mathrm{cod}(H)|=3$.
    Assume not.
	As $\mathrm{cod}(H)=\mathrm{cod}(G/V)\subseteq \mathrm{cod}(G)$ and $|\mathrm{cod}(G)|=4$,
	we have $\mathrm{cod}(H)=\mathrm{cod}(G)$.
    Note that $r^{n}\mid a$ for some $a\in \mathrm{cod}(G)= \mathrm{cod}(H)$ 
	and that
	$r\neq q$,
    and so $r=p$ and $G$ is a $\{ p,q \}$-group.	
    Recall that $H/\mathbf{C}_{H}(V)$ is not a $q$-group,
	and so $|V|\geq p^{2}$.
	In fact, otherwise we conclude a contradiction that $p\nmid |H/\mathbf{C}_{H}(V)|$.
	If $|\mathrm{cod}(H/Q)|=2$, as $p^{2}\mid a$ for each $a \in \mathrm{cod}(G|V)\subseteq \mathrm{cod}(H)$,
    we conclude a contradiction by Corollary \ref{cor: height=2, |cod(G/N)|=2}.
	Note that $P>1$, and hence we may assume that $|\mathrm{cod}(H/Q)|> 2$.
    Now, we show that $|\mathrm{cod}(H/Q)|=3$.
	Indeed, otherwise $\mathrm{cod}(H/Q)=\mathrm{cod}(H)$;
	however, $q\mid a$ for some $a\in \mathrm{cod}(H)$ by Lemma \ref{lem: kernel, abelian, cod},
	a contradiction.
	Therefore, $H$ satisfies Hypothesis \ref{hy: height=2, |cod(G/N)|=3} and $\mathrm{cod}(H)=\mathrm{cod}(G)$.
    Recall that $|V|\mid \mathrm{cod}(\chi)\in \mathrm{cod}(G)=\mathrm{cod}(H)$ for each $\chi \in \mathrm{Irr}(G|V)$,
	and so Corollary \ref{cor: height=2, |cod(G/N)|=3} implies that 
	$\mathrm{cod}(\chi)=|V|$.
	Consider now the quotient group $\overline{G}:=G/\mathbf{C}_{H}(V)$.
    Note that $\overline{G}=\overline{V} \rtimes \overline{H}$ where $\overline{V}$
	is the unique minimal normal subgroup of $\overline{G}$,
	and so $\ker(\chi)=1$ for each $\chi\in \mathrm{Irr}(\overline{G}|\overline{V})$.
   As $\mathrm{Irr}(\overline{G}|\overline{V})\subseteq \mathrm{Irr}(G|V)$,
   we deduce that $\mathrm{cod}(\chi)=|V|=|\overline{V}|$ for each $\chi\in \mathrm{Irr}(\overline{G}|\overline{V})$.
  So, it follows that $\chi(1)=|\overline{H}|$ for each $\chi\in \mathrm{Irr}(\overline{G}|\overline{V})$.
   Therefore, $\overline{G}$ is a Frobenius group with complement $\overline{H}$ and kernel $\overline{V}$.
   Consequently, $H/\mathbf{C}_{H}(V)=\overline{H}$ is a $q$-group,
    a contradiction.

	Since $|\mathrm{cod}(H)|=3$ and $h(H)=2$,
	\cite[Corollary B]{qian2023} yields that $H$ is a Frobenius group with complement $P\cong \mathsf{C}_{p}$
	and elementary abelian kernel $Q$ such that $Q$ is a homogeneous $P$-module over $\mathbb{F}_q$.
	In particular, $\mathrm{cod}(G/V)=\mathrm{cod}(H)=\{ 1,p,q^{d} \}$ where $q^{d}$ is the order of a $G$-chief factor in $K/V$.
    Set $C=\mathbf{C}_{H}(V)$.
	Then $C\unlhd G$.
	As $C$ is a proper normal subgroup of $H$, it follows that $C<Q$.
    Set $\overline{G}=G/C$.
    Then $\overline{H}$ is a Frobenius group with complement $\overline{P}\cong \mathsf{C}_{p}$
	and elementary abelian kernel $\overline{Q}$,
	and $\overline{V}$ is the unique minimal normal subgroup of $\overline{G}$.
    As $\mathbf{C}_{\overline{V}}(\overline{Q})=1$,
    it follows by \cite[Theorem 15.16]{isaacs1994} that $|\mathbf{C}_{\mathrm{Irr}(\overline{V})}(\overline{P})|=r^{m}$
	where $r^{pm}=r^{n}=|\overline{V}|$ and $m\geq 1$.
Let $\lambda\in \mathrm{Irr}(\overline{V})^\sharp$ and set $T=\mathrm{I}_{\overline{G}}(\lambda)$.
	As $\overline{G}=\overline{V} \rtimes \overline{H}$,
	the linear character $\lambda$ extends to some $\hat{\lambda}\in \mathrm{Irr}(T)$, 
	and so $\mathrm{Irr}(\overline{G}|\lambda)=\{ (\hat{\lambda} \alpha)^{\overline{G}}: \alpha \in \mathrm{Irr}(T/\overline{V}) \}$ by Gallagher's theorem and Clifford's theorem.
	Note that $(\hat{\lambda} \alpha)^{\overline{G}}$ has a trivial kernel, and so
	\[
		\mathrm{cod}((\hat{\lambda} \alpha)^{\overline{G}})=\frac{|\overline{G}|}{(\hat{\lambda} \alpha)^{\overline{G}}(1)}=\frac{|T|}{\alpha(1)}=|\overline{V}|\cdot \frac{|T/\overline{V}|}{\alpha(1)}=r^{n}\cdot \frac{|T/\overline{V}|}{\alpha(1)}.
	\] 
    Recalling that $|\mathrm{cod}(G)|=4$, $\mathrm{cod}(G/V)=\{ 1,p,q^{d} \}$ and that $n=pm\geq p$,
	we deduce that $\mathrm{cod}((\hat{\lambda} \alpha)^{\overline{G}})\notin \mathrm{cod}(G/V)$.
	Consequently, $\mathrm{cod}((\hat{\lambda} \alpha)^{\overline{G}})$
    is a constant for each $\lambda\in \mathrm{Irr}(V)^\sharp$ and each $\alpha\in \mathrm{Irr}(T/\overline{V})$.
	Therefore, $\alpha(1)=1$ for each $\alpha \in \mathrm{Irr}(T/\overline{V})$,
	and $T/\overline{V}$ is abelian.
	As $|\mathbf{C}_{\mathrm{Irr}(\overline{V})}(\overline{P})|=r^{m}>1$,
	$|\mathrm{I}_{\overline{G}}(\lambda)/\overline{V}|_p=|T/\overline{V}|_p=p$ for each $\lambda\in \mathrm{Irr}(\overline{V})^\sharp$.
    Since the elementary abelian $q$-group $\overline{Q}$ acts faithfully and completely reducibly on $\overline{V}$,
	it also acts faithfully and completely reducibly on $\mathrm{Irr}(\overline{V})$ by \cite[Lemma 1]{zhang2000}.
	So, an application of \cite[\S 19, Lemma 19.16]{huppertcharactertheory} yields that there exists an $\eta\in \mathrm{Irr}(\overline{V})^\sharp$ such that $\mathrm{I}_{\overline{Q}}(\eta)=1$.
    Therefore, $\mathrm{I}_{\overline{Q}}(\lambda)=1$ for every $\lambda\in \mathrm{Irr}(\overline{V})^\sharp$.
	As a consequence,
	$\overline{K}$ is a Frobenius group with cyclic complement $\overline{Q}\cong \mathsf{C}_{q}$ and elementary abelian kernel $\overline{V}$, and $\mathrm{cod}(G)=\{ 1,p,q,pr^{pm} \}$.
	 For each $\lambda\in \mathrm{Irr}(\overline{V})^\sharp$,
	since $|\mathrm{I}_{\overline{G}}(\lambda)/\overline{V}|=p$ and $\mathrm{I}_{\overline{H}}(\lambda)\in \mathrm{Syl}_{p}(\overline{H})$,
	we have
	$$r^{pm}-1=|\mathrm{Irr}(\overline{V})|-1=|\mathrm{Syl}_p(\overline{H})|\cdot 
	(|\mathbf{C}_{\mathrm{Irr}(\overline{V})}(P)|-1)=q(r^m-1),$$
    which gives $q=\frac{r^{pm}-1}{r^{m}-1}$.
    An application of Lemma \ref{lem: 2-Frobenius} to $\overline{G}$ now shows that $\overline{V}$
	is minimal normal in $\overline{K}$.

	Finally, we show $C=1$. 
    Assume not.
    Then $C$ is a nontrivial normal subgroup of the Frobenius group $H$ 
	(with complement $P\cong \mathsf{C}_{p}$ and elementary abelian kernel $Q$).
    Note that $K=V \rtimes Q$ has Fitting height 2, 
	and so $C$ is a proper subgroup of the elementary abelian $q$-group $Q$.
    Applying Lemma \ref{lem: kernel, abelian, cod},
	we deduce that $qr^{pm}=q|V|\mid a$ for some $a\in \mathrm{cod}(G)$
	which contradicts $\mathrm{cod}(G)=\{ 1,p,q,pr^{pm} \}$.
\end{proof}

\begin{thm}\label{thm: height=3}
	Assume Hypothesis \ref{hy: height=3}.
	Then $|\mathrm{cod}(G)|=4$
	if and only if $G/V$ is a Frobenius group with complement of order $p$ and cyclic kernel $K/V$
	of order $q=\frac{r^{pm}-1}{r^m-1}$,
	and $K$ is a Frobenius group with elementary abelian kernel $V$ of order $r^{pm}$ such that $V$ is minimal normal in $K$.	
\end{thm}
\begin{proof}
	We assume first that $|\mathrm{cod}(G)|=4$.
    Let $G$ be a counterexample of minimal possible order.
    By Lemma \ref{lem: height=3 cod=4 V minimal}, $V$ is not minimal normal in $G$.
	Let $D$ be a minimal normal subgroup of $G$ in $V$.
	Then $h(G/D)=h(G)=3$.
    As $|\mathrm{cod}(G/D)|=4$ by Lemma \ref{lem: height=3, induction},
    the minimality of $G$ then implies that
	$G/V$ is a Frobenius group with complement of order $p$ and kernel $K/V$
	of order $q=\frac{r^{pm}-1}{r^m-1}$,
	and $K/D$ is a Frobenius group with abelian kernel $V/D=\mathbf{F}(G/D)$ of order $r^{pm}$ such that $V/D$ minimal normal in $K/D$. 
    Moreover, $\mathrm{cod}(G)=\mathrm{cod}(G/D)=\{ 1,p,q,pr^{pm} \}$.
	As $\mathbf{F}(G)D/D\leq \mathbf{F}(G/D)= V/D$ and $V$ is nilpotent, it forces $\mathbf{F}(G)= V$.

	Suppose that $G$ has another minimal normal subgroup, say $E$.
	Then $E\leq \mathbf{F}(G)=V$.
    Repeating the preceding argument with $E$ in place of $D$,
    we obtain that 
	 $K/E$ is a Frobenius group with elementary abelian kernel $V/E$ of order $\ell^{pm}$ for some prime $\ell$, and $q=\frac{\ell^{pm}-1}{\ell^{m}-1}$.
	Since the function $f(x)=\frac{x^{p}-1}{x-1}$ is a strictly increasing function for $x>0$,
    the equality $\frac{\ell^{pm}-1}{\ell^{m}-1}=\frac{r^{pm}-1}{r^{m}-1}$ forces
    $\ell^{m}=r^{m}$, hence $\ell=r$.
    So, $K$ is a Frobenius group with elementary abelian kernel $V=D\times E$ of order $r^{2pm}$ and complement $Q\cong \mathsf{C}_{q}$.
	Set $H=\mathbf{N}_{G}(Q)$.
	Then the Frattini's argument implies that $G=VH$ where $V\cap H=\mathbf{C}_{V}(Q)=1$.
	So, $H\cong G/V$ is a Frobenius group with complement $P\cong \mathsf{C}_{p}$ and kernel $Q\cong \mathsf{C}_{q}$ where $q=\frac{r^{pm}-1}{r^{m}-1}$.
	Let $x$ be a nontrivial element in $Q$.
	Then $P \cap P^{x}=1$.
    As $|\mathbf{C}_{\mathrm{Irr}(D)}(P)|>1$ and $|\mathbf{C}_{\mathrm{Irr}(E)}(P^{x})|>1$ by \cite[Theorem 15.16]{isaacs1994},
	we take $\delta\in \mathbf{C}_{\mathrm{Irr}(D)}(P)^\sharp$ and $\epsilon \in \mathbf{C}_{\mathrm{Irr}(E)}(P^{x})^\sharp$ and set $\lambda=\delta\times \epsilon$.
    Then $\lambda\in \mathrm{Irr}(V)$ such that
	$\mathrm{I}_{H}(\lambda)=\mathrm{I}_{H}(\delta)\cap \mathrm{I}_{H}(\epsilon)=P\cap P^{x}=1$, and hence $\mathrm{I}_{G}(\lambda)=V$.
	Consequently, $\lambda^{G}\in \mathrm{Irr}(G)$ and
	$\ker(\lambda^{G})\leq V$.
	By calculation,
	$$\mathrm{cod}(\lambda^{G})=\frac{|G:\ker(\lambda^{G})|}{\lambda^{G}(1)}=\frac{pq\cdot |V|}{pq\cdot |\ker(\lambda^{G})|}=\frac{|V|}{|\ker(\lambda^{G})|}\in \{ r^{pm}, r^{2pm} \}.$$
    As 
	a consequence, we conclude a contradiction that $\mathrm{cod}(\lambda^{G})\notin \mathrm{cod}(G)=\{ 1,p,q,pr^{pm} \}$.

    So, $D$ is the unique minimal normal subgroup of $G$.
	It follows that $V=\mathbf{F}(G)$ is a $r$-group.
     Let $\chi\in \mathrm{Irr}(G|D)$, and observe that $\ker(\chi)\cap D=1$.
	Then $\ker(\chi)=1$.
	In particular, $\mathrm{cod}(\chi)=|G|/\chi(1)$.
	As $D$ is an abelian minimal normal subgroup of $G$,
    it follows by Lemma \ref{lem: kernel, abelian, cod} that $|D|\mid \mathrm{cod}(\chi)$.
	Recall that $\mathrm{cod}(G)=\{ 1,p,q,pr^{pm} \}$.
	If $\mathrm{cod}(\chi)=p$, 
	then $|G|<\mathrm{cod}(\chi)^{2}=p^{2}<pqr^{pm}|D|=|G|$, a contradiction.
	So, 
	we deduce that $\mathrm{cod}(\chi)=pr^{pm}$.
    Since $pqr^{pm}|D|=|G|<\mathrm{cod}(\chi)^{2}=p^{2}r^{2pm}$,
	we have that $|D|<r^{pm}\cdot \frac{p}{q}<r^{pm}=|V/D|$ where the second inequality holds as $G/V$ is a Frobenius group with complement of order $p$ and kernel of order $q$.
	Also, as
	$$\frac{pqr^{pm}\cdot |D|}{\chi(1)}=\frac{|G|}{\chi(1)}=pr^{pm},$$
    $\chi(1)=q|D|$.	
	Let $\theta$ be an irreducible constituent of $\chi_V$.
    Note that $\mathrm{I}_{K}(\theta)=V$ as $\chi(1)=q|D|$,
	and that $K/D$ is a Frobenius group with kernel $V/D$.
	So, $K$ is a Frobenius group with kernel $V$ and complement $Q\cong \mathsf{C}_{q}$.
	An application of part (1) of Lemma \ref{lem: size of chief factor in Frob ker} to $K$ yields that 
	every $K$-chief factor shares the same order with the $K$-chief factor $V/D$.
	Consequently, $|D|\geq |V/D|$, a contradiction.

Conversely, we assume that $G/V$ is a Frobenius group with complement of order $p$ and cyclic kernel $K/V$
	of order $q=\frac{r^{pm}-1}{r^m-1}$,
	and $K$ is a Frobenius group with elementary abelian kernel $V$ of order $r^{pm}$ such that $V$ is minimal normal in $K$.
	Then $G=V \rtimes H$ where $H$ is a Frobenius group with complement $P\cong \mathsf{C}_{p}$ and kernel $Q\cong \mathsf{C}_{q}$, and $VQ$ is a Frobenius group with kernel $V$ such that $V$ is minimal normal in $VQ$.
	Also, it is routine to check that 
    $\mathrm{cod}(G/V)=\{ 1,p,q \}$.
	Note that $\mathbf{C}_{\mathrm{Irr}(V)}(Q)=1$,
	and hence \cite[Theorem 15.16]{isaacs1994} yields that 
    $|\mathbf{C}_{\mathrm{Irr}(V)}(P)|^{p}=|\mathrm{Irr}(V)|=|V|=r^{pm}$.
	Hence, $|\mathbf{C}_{\mathrm{Irr}(V)}(P)|=r^{m}$.
	Noting also that, for each $x\in Q^\sharp$, 
	$$\mathbf{C}_{\mathrm{Irr}(V)}(P)\cap \mathbf{C}_{\mathrm{Irr}(V)}(P^x)=\mathbf{C}_{\mathrm{Irr}(V)}(\langle P,P^x\rangle)=\mathbf{C}_{\mathrm{Irr}(V)}(H)=1,$$ 
    we deduce that $\mathrm{Irr}(V)=\bigcup_{x\in Q} \mathbf{C}_{\mathrm{Irr}(V)}(P^{x})$ by calculating the sizes of both sets.
    In other words, $|\mathrm{I}_{H}(\lambda)|=p$ for each $\lambda\in \mathrm{Irr}(V)^\sharp$.
    As $\mathrm{I}_{G}(\lambda)/V\cong \mathrm{I}_{H}(\lambda)\cong \mathsf{C}_{p}$,
	$\lambda$ extends to $\mathrm{I}_{G}(\lambda)$ and every extension of $\lambda$ in $\mathrm{I}_{G}(\lambda)$ is linear.
	So, Clifford's correspondence yields that $\chi(1)=q$ for each $\chi\in\mathrm{Irr}(G|\lambda)$.
	As $V=\mathbf{F}(G)$ and $\ker(\chi)\cap V=1$,
	we have $\ker(\chi)=1$, and
	therefore $\mathrm{cod}(\chi)=pr^{pm}$. 
    In all, $\mathrm{cod}(G)=\{ 1,p,q,pr^{pm} \}$.
\end{proof}

\begin{rmk}
  {\rm    We present some examples of solvable groups $G$ with Fitting height 3 such that 
	$|\mathrm{cod}(G)|=4$.
	The smallest example is $G=\mathsf{S}_{4}=V \rtimes H$ where $V\cong (\mathsf{C}_{2})^{2}$ and $H=Q \rtimes P\cong \mathsf{S}_{3}$.
    There are also some examples demonstrating that $r$ is not necessary equal to $p$.
	For instance, 
	$G=V \rtimes H$ where $V\cong (\mathsf{C}_{2})^{9}$ and $H$ is a Frobenius group with complement $P\cong \mathsf{C}_{3}$ and kernel $Q\cong \mathsf{C}_{73}$ such that $V$ is a faithful irreducible $Q$-module. }
\end{rmk}

\section{Main results}

\begin{lem}\label{lem: nilpotent}
  Let $G$ be a solvable group and $N$ the nilpotent residual of $G$.
  Assume that $|\pi(G/N)|>1$.
  Then $|\mathrm{cod}(G)|=4$ if and only if $G$ is a direct product of an elementary abelian $p$-group
  and an elementary abelian $q$-group where $p\neq q$.
\end{lem}
\begin{proof}
  Assume that $G=P\times Q$ where $P$ is a nontrivial elementary abelian $p$-group, $Q$ is 
  a nontrivial elementary abelian $q$-group, and $p\neq q$.
  As $\mathrm{cod}(P)=\{ 1,p \}$ and $\mathrm{cod}(Q)=\{ 1,q \}$,
  we conclude by part (3) of Lemma \ref{lem: direct product, kernel and codegrees} that $\mathrm{cod}(G)=\{ 1,p,q,pq \}$.

  Assume that $|\mathrm{cod}(G)|=4$.
  Since $G/N$ is a nilpotent group with $|\pi(G/N)|>1$,
  it follows by part (3) of Lemma \ref{lem: direct product, kernel and codegrees} that $G/N$ is a direct product of an elementary abelian $p$-group
  and an elementary abelian $q$-group where $p\neq q$.
  Therefore, $N=G'$ and $\mathrm{cod}(G)=\mathrm{cod}(G/N)=\{ 1,p,q,pq \}$.
  
  We next show that $N=1$.
  Let $G$ be a counterexample of minimal possible order such that $N>1$.
 Then $N$ is minimal normal in $G$.
  By Lemma \ref{lem: kernel, abelian, cod}, $|N|\mid \mathrm{cod}(\chi)$ for some $\chi \in \mathrm{Irr}(G)$.
  Therefore, $|N|=p$ or $q$.
  Without loss of generality,
  we may assume that $|N|=p$.
  Then $G=P  \rtimes Q$ where $P\in \mathrm{Syl}_{p}(G)$ and $Q\in \mathrm{Syl}_{q}(G)$.
  Note that $P/N$ is elementary abelian,
  and so $\Phi(P)\leq N$.
  Since $Q$ acts coprimely on $P$, we conclude that $P$ is elementary abelian.
  In fact, otherwise, $\Phi(P)=N$; as $Q$ fixes every element in $P/\Phi(P)$, $[P,Q]=1$, a contradiction.
  So, $G=(N\times P_0) \rtimes Q$ where $P_0\unlhd G$.
  Let $C=\mathbf{C}_{Q}(N)$ and $Q_0$ a complement of $C$ in $Q$.
  Then $NQ_0$ is a Frobenius group with complement $Q_0\cong \mathsf{C}_{q}$ and kernel $N\cong \mathsf{C}_{p}$.
  Also, $G= NQ_0 \times (P_0\times C)$.
  Let $\theta$ be a nonlinear character in $\mathrm{Irr}(NQ_0)$ and $\beta\in \mathrm{Irr}(P_0)^\sharp$.
  Then $\chi=\theta \times (\beta\times 1_C) \in \mathrm{Irr}(G)$.
  Since $NQ_0$ is a Frobenius group and $\ker(\theta)=1$,
  $\mathbf{Z}(\theta)=\mathbf{Z}(NQ_0)=1$.
  So, it follows by part (2) of Lemma \ref{lem: direct product, kernel and codegrees} that 
  $\mathrm{cod}(\chi)=\mathrm{cod}(\theta)\mathrm{cod}(\beta\times 1_C)=\mathrm{cod}(\theta)\mathrm{cod}(\beta)\mathrm{cod}(1_C)=p^{2}$,
  a contradiction.
\end{proof}

Finally, we are able to prove Theorem \ref{thmA}.

\begin{proof}[Proof of Theorem \ref{thmA}]
	Assume that $|\mathrm{cod}(G)|=4$.
	If $G$ is nonsolvable, then \cite[Theorem]{liu2021} yields that $G\cong \mathrm{SL}_2(2^{f})$ where $f\geq 2$, i.e. (7) holds.
	Suppose that $G$ is solvable.
	Then $h(G)\leq 3$ by Lemma \ref{lem: qz}.
	Let $N$ be the nilpotent residual of $G$.
   If $|\pi(G/N)|\geq 2$, then part (1) holds by Lemma \ref{lem: nilpotent}.
	Assume next that $G/N$ is a $p$-group.
	If $h(G)=2$, then one of (2), (3), (4) or (5) holds by Theorems \ref{thm: height=2, |cod(G/N)|=3, classification} and \ref{thm: height=2, |cod(G/N)|=2, classification} and Proposition \ref{prop: height=2, |cod(G/N)|=2, C>1}.
	If $h(G)=3$, then part (6) holds by Theorem \ref{thm: height=3}. 
	
	Conversely, assume that one of (1), (2), (3), (4), (5), (6) or (7) holds. 
	Then we are done by Lemma \ref{lem: nilpotent}, Proposition \ref{prop: height=2, |cod(G/N)|=2, C>1},
	Theorems \ref{thm: height=2, |cod(G/N)|=3, classification}, \ref{thm: height=2, |cod(G/N)|=2, classification} and \ref{thm: height=3}, and \cite[Theorem]{liu2021}.
\end{proof}

For solvable groups $G$, the Isaacs-Seitz conjecture asserts that the derived length of $G$ is bounded by
the size of the set of character degrees, i.e.
$\mathrm{dl}(G)\leq |\mathrm{cd}(G)|$. 
This conjecture is still open in general, but it was settled
when $|\mathrm{cd}(G)|\leq 4$ (see \cite[Theorem 12.15]{isaacs1994} and \cite{garrison}).
A ``dual'' question arises concerning the set of character codegrees: for each solvable group $G$,
\begin{center}
	$\mathrm{dl}(G)\leq |\mathrm{cod}(G)|$.
\end{center}
Due to \cite{alizadeh2019,du2016}, this inequality can be verified when $|\mathrm{cod}(G)|\leq 3$.
However, the case $|\mathrm{cod}(G)|=4$ has remained stubbornly difficult.
In fact, we are only able to show the following result.

\begin{cor}
	Let $G$ be a solvable group with $|\mathrm{cod}(G)|=4$.
	Then one of the following holds.
	\begin{description}
		\item[(1)] $\mathrm{dl}(G)\leq |\mathrm{cod}(G)|$.
		\item[(2)] $|\mathrm{cd}(G)|>|\mathrm{cod}(G)|$, and 
		either $G$ is a group of prime-power order or $G$ satisfies (2f) of Theorem \ref{thmA}.
	\end{description}
\end{cor}
 \begin{proof}
     If $|\mathrm{cd}(G)|\leq |\mathrm{cod}(G)|$, then we are done by \cite[Theorem 12.15]{isaacs1994} and \cite{garrison}.
	 So, we may assume that $|\mathrm{cd}(G)|>|\mathrm{cod}(G)|$ and that $G$ is a group of non-prime-power order 
	 which does not satisfy (2f) of Theorem \ref{thmA}.
	 Applying Theorem \ref{thmA}, we conclude that $\mathrm{dl}(G)\leq |\mathrm{cod}(G)|$.
 \end{proof}

\begin{acknowledgement}
	The authors are grateful to the referee for her/his
 valuable comments.
\end{acknowledgement}


\begin{thebibliography}{999}\setlength{\itemsep}{-2mm} 
\small
	\bibitem{alizadeh2019}
	F. Alizadeh, H. Behravesh, M. Ghaffarzadeh, M. Ghasemi and S. Hekmatara,
	\newblock  Groups with few
	codegrees of irreducible characters,
	\newblock {\em Comm. Algebra} {\bf47(3)} (2019), 1147--1152.

	\bibitem{alizadeh2022}
	F. Alizadeh, H. Behravesh, M. Ghaffarzadeh, M. Ghasemi and S. Hekmatara,
	\newblock Corrigendum to: Groups with few codegrees of irreducible characters,
	\newblock {\em Comm. Algebra} {\bf50(12)} (2022), 5479--5480.





	\bibitem{croome2020}
	S. Croome and M. Lewis,
	\newblock $p$-groups with exactly four codegrees,
	\newblock \emph{J. Group Theory} {\bf 23} (2020), 1111--1122.




	\bibitem{doerk1992}
	K. Doerk and T.O. Hawkes,
	\newblock {\em Finite soluble groups},
	\newblock Walter de Gruyter, Berlin, 1992.

	\bibitem{du2016}
	N. Du and M. Lewis,
	\newblock  Codegrees and nilpotence class of $p$-groups,
	\newblock {\em J. Group Theory} {\bf19(4)} (2019), 561--568.

	\bibitem{gap} The GAP Group, GAP - Groups, Algorithms, and Programming, Version 4.13.1, 2024, \href{http://www.gap-system.org}{http://www.gap-system.org}.


		\bibitem{garrison}
	S. Garrison,
	\newblock  On groups with a small number of character degrees,
	\newblock Ph.D. Thesis, University of Wisconsin, Madison, 1973.




	\bibitem{harris1977}
	M. Harris,
	\newblock On $p'$-automorphisms of abelian $p$-groups,
	\newblock {\em Rocky Mountain J. Math.} {\bf7} (1977), 751--752.

	\bibitem{higman1963} 
	G. Higman,
	\newblock Suzuki 2-groups,
	\emph{Illinois J. Math.} {\bf 7(1)} (1963), 79--96.
	
	\bibitem{huppertgrouptheory}
	B. Huppert,
	\newblock \emph{Endliche gruppen I},
	\newblock Springer-Verlag, Berlin, 1967.



   \bibitem{huppertbook2}
   B. Huppert, N. Blackburn,
   \emph{Finite groups II},
    Springer-Verlag, Berlin, 1982.


	\bibitem{huppertcharactertheory}
	B. Huppert.
	\newblock {\em Character Theory of Finite Groups}.
	 \newblock Walter DeGruyter, Berlin, 1998.




	 

   \bibitem{isaacs1994}
   I.M. Isaacs,
   \newblock {\em Character theory of finite groups},
   \newblock Dover, New York, 1994.


   \bibitem{isaacsgrouptheory}
   I.M. Isaacs,
   \newblock {\em Finite group theory},
   \newblock  American Mathematical Society, Providence, RI, 2008.





   \bibitem{isaacspassman1966}
	I.M. Isaacs and D.S. Passman,
	\newblock Half-transitive automorphism groups,
	\newblock {\em Canad. J. Math.} {\bf18} (1966), 1243--1250.






	\bibitem{lewis2017}
	M.L. Lewis,
	\newblock Semi-extraspecial groups,
	\newblock {\em Southern Regional Algebra Conference}, Springer, Cham, (2017), 219--237.

	\bibitem{liang16}
	D. Liang and G. Qian,
	\newblock Finite groups with coprime character degrees and codegrees,
	\newblock \emph{J. Group Theory} {\bf 19(5)} (2016), 763--776.


	\bibitem{liu2025}
	Y. Liu and X. Song,
	\newblock Solvable groups with four character codegrees,
	\newblock {\em J. Algebra Appl.}, online; 
    \href{https://doi.org/10.1142/S0219498827500368}{https://doi.org/10.1142/S0219498827500368}.

	\bibitem{liu2021}
	Y. Liu and Y. Yang,
	\newblock  Nonsolvable groups with four character codegrees,
	\newblock {\em J. Algebra Appl.} {\bf20(11)} (2021), 2150196.






	\bibitem{manzwolf1992} 
	O. Manz and T.R. Wolf.
	\newblock {\em Representations of Solvable Groups}.
	\newblock Cambridge University
   Press, Cambridge, 1992.





   \bibitem{moreto2022}
   A. Moret\'o,
   \newblock Character degrees, character codegrees and nilpotence class of $p$-groups,
   \newblock \emph{Comm. Algebra} {\bf 50(2)} (2022), 803--808.





	\bibitem{qian2002}
	G. Qian,
	\newblock Notes on character degrees quotients of finite groups,
	\newblock \emph{J. Math.(in Chinese)} {\bf 22(2)} (2002), 217--220.

	
 

	\bibitem{qian2007}
	G. Qian, Y. Wang and H. Wei,
	\newblock Co-degrees of irreducible characters in finite groups,
	\newblock \emph{J. Algebra} {\bf 312(2)} (2007), 946--955.

	



	


	\bibitem{qian2023}
	G. Qian and Y. Zeng,
	\newblock Finite groups in which all nonlinear irreducible characters have the same codegree,
	\newblock {\em Comm. Algebra} {\bf51(3)} (2023), 1293--1297.
	
	

	\bibitem{qian2025}
	G. Qian and Y. Zeng,
	\newblock Character codegrees, kernels, and Fitting heights of solvable groups,
	\href{https://arxiv.org/pdf/2502.01950}{https://arxiv.org/pdf/2502.01950}.

	



	\bibitem{zhang2000}
	J. Zhang,
	\newblock A note on character degrees of finite solvable groups,
	\newblock \emph{Comm. Algebra} {\bf 28(9)} (2000), 4249--4258.



\end{thebibliography}
\end{document}